\documentclass[10pt]{amsart}
\usepackage{FancyMath}

\title[Koszul dual algebras part 2]{Koszul dual $\A_{\infty}$-algebras from star-shaped diagrams -- part 2}
\author[Isabella Khan]{Isabella Khan}
\date{\today}

\begin{document}
	\begin{abstract}
		This paper proves a Koszul duality result between weighted $\A_{\infty}$-algebras constructed in the author's previous work. In the process, we construct a new box tensor product for weighted $\A_{\infty}$ bimodules, and verify a correspondence between weighted $\A_{\infty}$-algebra maps and a particular class of $\A_{\infty}$-bimodule.
	\end{abstract}
	\maketitle\blfootnote{The author was partially supported by a NSF Graduate Research Fellowship, and by the Simons Collaboration on New Structures in Low Dimensional Topology.}
	\tableofcontents
	
	\thispagestyle{empty}
	
	\section{Introduction}

    The notion of Koszul duality arises frequently in both algebra and, more recently, in low-dimensional topology. From the abstract representation theoretic perspective it can be viewed as a certain type of equivalence of categories, as in the work of Francis-Gaitsgory on chiral Lie algebras~\cite{FrGa}, or  recently~\cite{Heuts}, in which Heuts proves an operadic equivalence, disproving a conjecture from~\cite{FrGa}. In the context of low-dimensional topology, Koszul duality arises in the study of algebraic structures relating to Heegaard Floer homology and knot Floer homology, providing  a notion of algebraic equivalence which accurately reflects the holomorphic curve theory involved in these constructions. For instance, in the work of Ozsv\'ath and Szab\'o,~\cite{Pong1},     the authors prove that the $\A_{\infty}$-structure induced by a pair of algebras from bordered knot Floer homology introduced in~\cite{AlgMat} satisfies a Koszul duality relation. In~\cite{LOTtorus}, Lipshitz, Ozsv\'ath and Thurston prove that the analogous algebras from the $HF^{-}$-version of bordered Heegaard Floer homology are Koszul dual to themselves, and in~\cite{ZemkeKos}, Zemke produces Koszul dual $\A_{\infty}$-algebras which encode the Dehn surgery formulae from Heegaard Floer homology. 
    
    In~\cite{KhKos}, the author constructed a family of weighted $\A_{\infty}$-algebras motivated by bordered knot Floer homology, which corresponded to a particular family of handle decompositions of a punctured disk. The author also constructed in~\cite{KhKos} a pair of $\A_{\infty}$-bimodules which encode the holomorphic geometry of these  handle decompositions. The current paper builds on these results to show that the $\A_{\infty}$-algebras and -bimodules constructed in~\cite{KhKos} satisfy a Koszul duality relation: 

    \begin{theorem}\label{Koszul}
        Let $\A$ and $\B$ be the weighted $\A_{\infty}$-algebras defined in Section~\ref{algs} (and in~\cite{KhKos}). Then there exists a DD-bimodule $\:^{\A} X^{\B}$, constructed in Section~\ref{dd} and a weighted AA-bimodule $\:_{\B} Y_{\A}$ constructed in Section~\ref{aa} such that 
        \begin{equation}\label{diag4}
            \:^{\A} X^{\B} \boxtimes \:_{\B} Y_{\A} \cong \:^{\A}\id_{\A}
        \end{equation}
        and
        \begin{equation}\label{diag5}
            \:^{\B} X^{\A} \boxtimes \:_{\A} Y_{\B} \cong \:^{\B}\id_{\B}
        \end{equation}
        where $\:^{\A} \id_{\A}$ and $\:^{\B} \id_{\B}$ are the identity DA-bimodules over $\A$ and $\B$ respectively (as in Definition~\ref{aux1}), and $\boxtimes$ is as in Definition~\ref{box1}. 
    \end{theorem}

    \begin{wrapfigure}{r}{4cm}\vspace{-25pt}
        \includegraphics[width = 4cm]{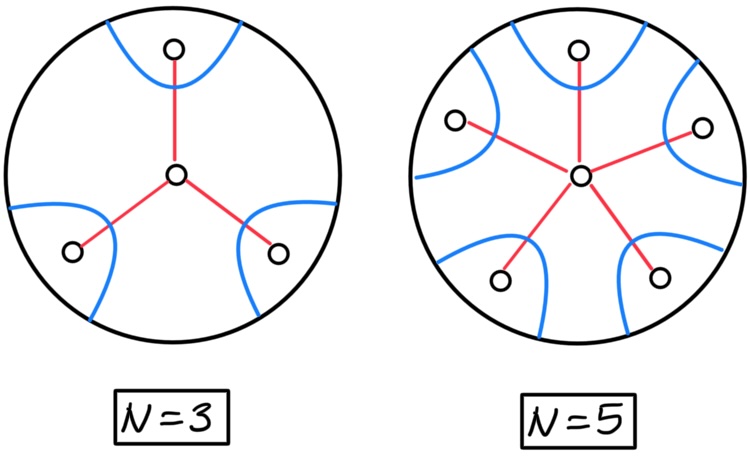}\vspace{-5pt}
        \caption{}\label{intro1}\vspace{-15pt}
    \end{wrapfigure}

    \subsection{Motivation}

    The algebras and bimodules from~\cite{KhKos} are largely geometric in their construction, and can be geometrically motivated in two ways. First, it is desirable to understand the algebra associated to an arbitrary handle decomposition of the disk. There are a number of existing constructions which can be viewed as steps in this direction, for instance the results of Ozsv\'ath and Szab\'o in~\cite{Pong1}, those of Manion in~\cite{Man}, and those of Roberts in~\cite{Rob}. The algebras constructed in~\cite{KhKos} correspond to a different family of handle-decompositions of the multiply punctured disk, two examples of which are pictured in Figure~\ref{intro1}, and we expect that they may represent a step towards understanding algebras associated to a more general family of handle decompositions. 
    Notably, in~\cite{Pong1}, the authors use a duality result to understand more about the properties of algebras corresponding to a particular handle decomposition of the disk; the duality result Theorem~\ref{Koszul} gives insight into similar properties in the case of the decompositions considered here. Moreover, a general handle decomposition of the punctured disk takes the form of a tree, and the current paper and~\cite{KhKos} both explore new features of the algebras $\A$ and $\B$ which might be expected to be necessary in the general case. With all this in mind, it seems very reasonable that Theorem~\ref{Koszul} could be viewed as progress towards the goal of understanding the algebra associated to arbitrary handle decompositions of a disk. 

    \begin{wrapfigure}{r}{5cm}\vspace{-10pt}
        \includegraphics[width = 5cm]{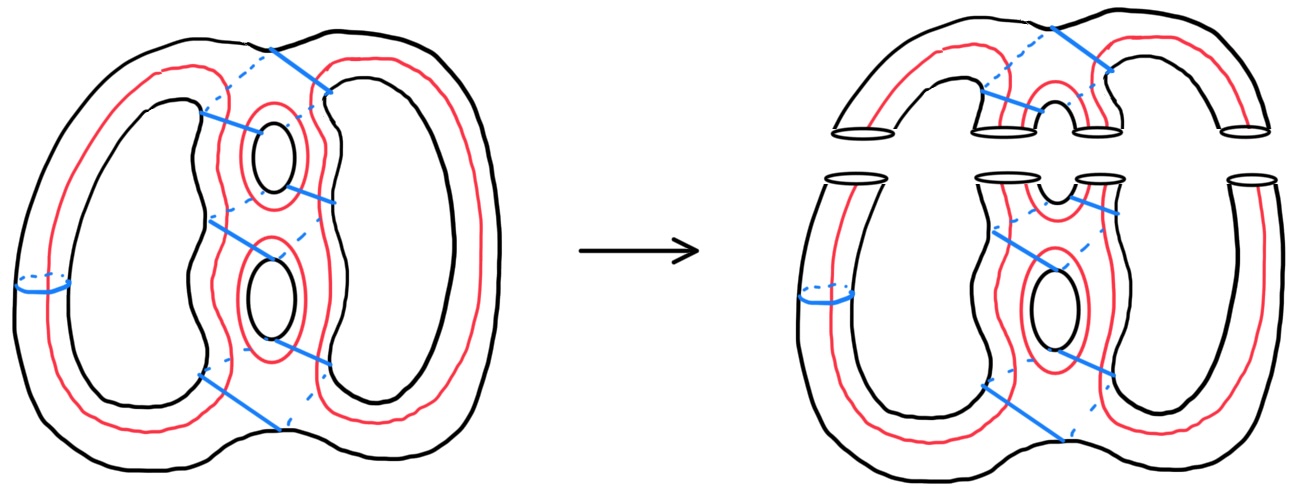}\vspace{-5pt}
        \caption{}\label{intro2}\vspace{-10pt}
    \end{wrapfigure}

    The results here and in~\cite{KhKos} are also motivated by the bordered knot Floer homology construction. Knot Floer homology is a topological invariant of knots in $S^3$ developed by Ozsv\'ath and Szab\'o in~\cite{OzSzKnot}, and, separately, by Rasmussen in~\cite{Ras}. It associates to any knot $K \subseteq S^3$ a bigraded abelian group $HFK^{\circ}(K)$, where ``$\circ$'' denotes one of the flavors ( $+, -$ or $\:\widehat\;$ ) of the invariant. The construction runs as follows: consider a projection of $K$ and a particular Heegaard diagram which encodes the data of $K$ in a particular sense -- for instance, for the left handed trefoil, we could consider the Heegaard diagram on the left hand side of Figure~\ref{intro2}. We then slice the diagram horizontally at a generic point to obtain two partial Heegaard diagrams, as on the right hand side of Figure~\ref{intro1}. To the horizontal slice, we associate an $\A_{\infty}$-algebra $\A_0$, and to each partial Heegaard diagram, we associate an $\A_{\infty}$-module such that when we take the tensor product over $\A_0$, we retrieve the knot Floer homology of $K$.

    \begin{wrapfigure}{r}{3cm}\vspace{-10pt}
        \begin{center}
        \includegraphics[width = 3cm]{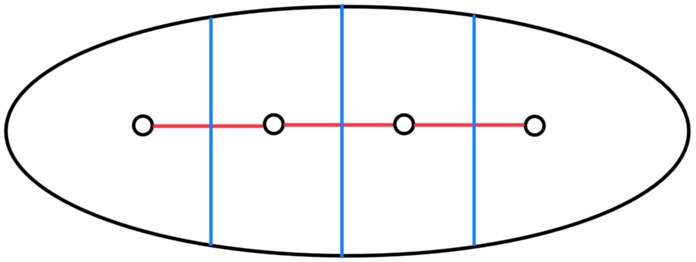}\end{center} \vspace{-10pt}
        \caption{}\label{intro3} \vspace{-15pt}
    \end{wrapfigure}
    
    The first step in any such construction is to define the algebra $\A_0$ associated to a given horizontal slice, which forms a linear graph such as the red portion of the diagram of Figure~\ref{intro3}. However, the algebras involved with the bordered $HFK^-$ flavor of knot Floer homology become very complicated; in order to understand these algebras better using a slightly easier algebra, Ozsv\'ath and Szab\'o construct in~\cite{Pong1} a ``pong algebra'' corresponding to the blue portion of Figure~\ref{intro3}.

    \begin{wrapfigure}{r}{4cm}\vspace{-25pt}
        \includegraphics[width = 4cm]{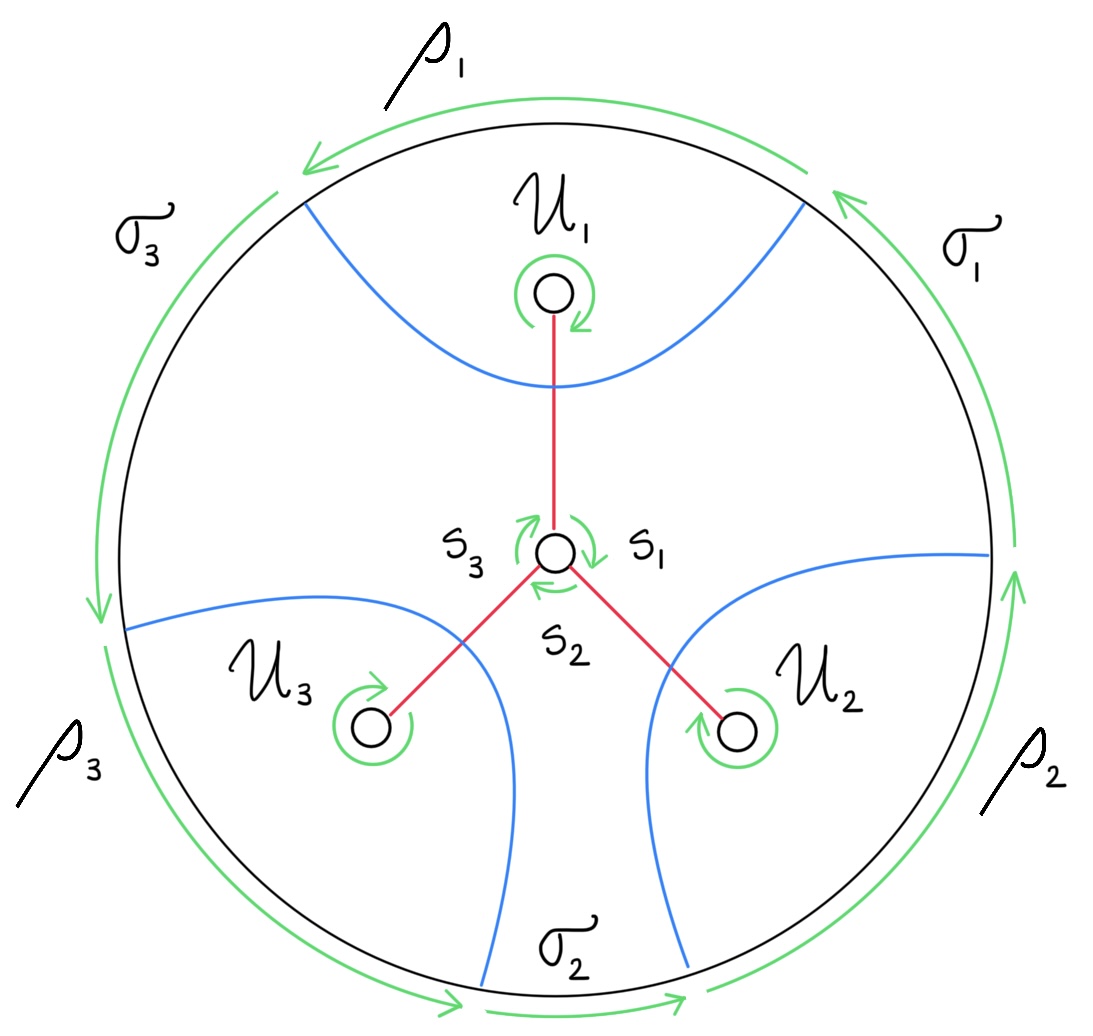}
        \caption{}\label{intro4} \vspace{-30pt}
    \end{wrapfigure}
    
    In the interest of understanding the algebras which arise from arcslide moves on these diagrams, it is desirable to construct algebras corresponding to a broader variety of graph. The higher-valence diagrams pictured in Figure~\ref{intro1}, which we call \emph{star diagrams} are a first step in this direction.

    \subsection{Geometric overview of the constructions from~\cite{KhKos}}

    We now give a brief overview of the construction of the algebras and bimodules involved in Theorem~\ref{Koszul}, with the aim of highlighting the geometric features of these constructions. Further specifics are given in Sections~\ref{algs} and~\ref{bims}, and for full details, see~\cite{KhKos}. This section assumes a certain amount of familiarity with the definitions of $\A_{\infty}$-algebras and bimodules; it may therefore be desirable to skip this section on a first reading, and return to it after reading Section~\ref{ainf}.
    
    The weighted $\A_{\infty}$-algebras $\A$ and $\B$ are $\A_{\infty}$-deformations of path algebras, whose generators correspond to arcs along the boundary circles of an $(N + 1)$-punctured disk called a \emph{star diagram}, two examples of which are pictured in Figure~\ref{intro1}. In Figure~\ref{intro4}, we illustrate the arcs which correspond to each generator for the case $N = 3$. The arcs labeled with $U_i$'s and $s_i$'s correspond to generators of $\A$, and the arcs labeled with $\rho_i$'s and $\sigma_i$'s correspond to generators of $\B$. Simple multiplication corresponds to concatenation of arcs. 
    
    In particular, the ground ring contains a subring $\F_2[\I_1, \ldots, \I_n]$, which we call the \emph{ring of idempotents} (since the $\I_i$ are defined so that $\I_i \I_j = 0$ if $i \neq j$ and $\I_i$ if $i = j$), which ensures that all non-zero operations make geometric sense. More specifically, note that each star diagram is decorated with $N$ red $\alpha$-arcs and $N$ blue $\beta$-arcs, which are labeled in clockwise order, starting at the 12 o'lock position. Each generator of $\A$ corresponds to an arc which starts and ends on some pair of $\alpha$-arcs, and each generator of $\B$ corresponds to an arc which starts and ends on some pair of $\beta$-arcs. The idempotents encode this information: we consider the case of $\A$ to illustrate this fact. Multiplication on $\A$ is defined such that for each generator $a \in \A$, there are unique $1 \leq i, j, \leq N$ with 
    \begin{equation}\label{intro5}
        \I_i \cdot a = a \cdot \I_j = a 
    \end{equation}
    and $\I_{i'} \cdot a = a \cdot \I_{j'} = 0$ for each $i' \neq i, \: j' \neq j$, and likewise for generators of $\B$. The $i, j$ such that~\eqref{intro5} holds are precisely those such that $a$ is an arc which starts on $\alpha_i$ and ends on $\alpha_j$. We call $i$ the \emph{initial idempotent} of $a$, and $j$ the \emph{final idempotent}, and construct operations such that $\mu_n^{\w} (a_1, \ldots, a_n)$ is always zero unless the initial idempotent of $a_k$ agrees with the final idempotent of $a_{k - 1}$, for each $2 \leq k \leq n$.
    
   Defining sensible higher operations is somewhat more complicated, but, in the case of $\A$, can also be understood geometrically. The higher operations on $\A$, are constructed by using the operad of $2N$-valent trees, which allows us to give explicit formulae for all of the infinitely many weighted and unweighted algebra operations. We illustrate a few of these in Figures~\ref{alg10} and~\ref{alg11}.

    All of these higher structures are motivated by the holomorphic curve theory of the decorated punctured disks which we call star diagrams. This is most apparent in the construction of the AA-bimodule $Y$ (see Section~\ref{aa} for details). The operations on this bimodule correspond to rigid holomorphic disks whose boundary contains a sequence of arcs corresponding to elements of $\A$ and $\B$, and the structure relations on this module correspond to 1-dimensional moduli spaces of disks, whose ends cancel in pairs. See Figure~\ref{aa17} for examples of disks corresponding to non-zero AA-bimodule operations. 
    
    Note in particular that the notion of \emph{weighted operations} -- that is, $\mu_n^{\w}$ with $\w \neq0$ -- which appear in this bimodule correspond to the inclusion of rigid holomorphic disks which pass entirely over one or more boundary circles. See for example the disk on the right side of Figure~\ref{aa17}. That these weighted operations introduce more algebraic complexity is unsurprising, given that moduli spaces of punctured disks are more difficult to understand than those of unpunctured disks. 

    \subsection{The structure of this paper}

    The primary aim of the current paper is to provide the proof of Theorem~\ref{Koszul}. Section~\ref{defs} gives the relevant definitions; many of these are standard, and for these, our exposition largely follows~\cite{DiagBible}. Several constructions, notably those in Sections~\ref{prim} and~\ref{boxy}, are original to this paper, as is Proposition~\ref{aux5}, a fact about weighted $\A_{\infty}$-algebra homomorphisms given in Section~\ref{ainf}, which is crucial to the proof of Theorem~\ref{Koszul}. In Section~\ref{algs}, we briefly outline the constructions of the $\A_{\infty}$-algebras and bimodules from~\cite{KhKos}, involved in Theorem~\ref{Koszul}. For the full constructions, see~\cite{KhKos}. In the final section, we give the proof of the main theorem. 

    Please note that a statement of Theorem~\ref{Koszul} appeared in a previous version of~\cite{KhKos}. However, the proof given there was incomplete, and~\cite{KhKos} has since been revised to contain only the constructions of the algebras and bimodules involved with Theorem~\ref{Koszul}. The current paper now contains the necessary algebraic framework, as well as a corrected and complete proof of Theorem~\ref{Koszul}. 

   \subsection*{Acknowledgments} I would like to thank Peter Ozsv\'ath for many helpful conversations during the preparation of this paper, and Robert Lipshitz, whose comments on my PhD thesis motivated the results given below. I would also like to acknowledge the anonymous referee whose comments on~\cite{KhKos} were very helpful in writing this paper.

    \section{Definitions}\label{defs}
        \subsection{Trees}

        We next define the space of \emph{stably weighted trees,} $X^{*,*}$. First, define a \emph{marked tree} to be a rooted planar tree $T$ with a subset of leaves called the \emph{inputs} of $T$ and a single leaf called the \emph{output} leaf of $T$. Define a \emph{popsicle} to be a leaf of a marked tree $T$ which is not an input or an output, and define a vertex of $T$ to be \emph{internal} if it is not an input or an output; hence internal vertices are popsicles or vertices at which $T$ has valence $> 1$. 

        A \emph{weighted tree} is a marked tree $T$ together with a weight function $\w$ from the vertices of $T$ to some lattice of weights $\Lambda = \Z_{\geq 0}\langle \e_1, \ldots, \e_n \rangle$. For $\w = \sum_i k_i \e_i \in \Lambda$ define
        \begin{equation}\label{defs13}
            |\w| = \sum_{i} k_i,
        \end{equation}
        and say that for $\vv = \sum_i \ell_i \e_i \in \Lambda$
        \begin{equation}\label{defs14}
            \vv < \w \leftrightarrow \ell_i < k_i \text{ for each } i, 
        \end{equation}
        and similarly for $\vv \leq \w$. We say that a vertex $v$ of a weighted tree $T$ is \emph{stable} if either the valence of $T$ at $v$ is $> 2$, or $\w(v) > 0$ (with ordering conventions as in~\eqref{defs13}). A weighted tree $T$ is called a \emph{stably weighted tree} if every internal vertex of $T$ is stable. Define the \emph{total weight} of a tree $T$ to be the sum
        \[
            \w = \sum_{v \text{ a vertex of } T} \w(v).
        \]

        Let $\TT_{n,\w}$ be the set of stably weighted trees with $n$ inputs and total weight $\w$. Define a dimension function on $\TT_{n, \w}$ by:
        \begin{equation}\label{defs12}
            \dim T = n + 2|\w|-v-1
        \end{equation}
        We define a differential on elements of $\TT_{n,\w}$ in the following way. Consider a pair $(S,e)$, where $S$ is a stably weighted tree and $e$ is an edge of $e$. Consider another stably weighted tree $T$. We say that $(S,e)$ is an \emph{edge expansion} of $T$ if and only if $T$ can be obtained from $S$, by contracting the edge $e$ into a single vertex. In this case, we also say that $S$ is obtained from $T$ by inserting an edge. 

        \begin{definition}\label{defs15}
            The chain complex $X^{*,*}_*$ of stably weighted trees
        \end{definition}
        
        \noindent Fix a ground ring $R$. Define the complex $X^{n, \w}_*$ of stably weighted trees with $n$ inputs and total weight $\w$ to be the $R$-module generated by $\TT_{n, \w}$, graded by dimension (defined as in~\eqref{defs12}), so that
        \[ 
            X^{n,\w}_k = R \langle T \in \TT_{n, \w} : \dim T = k \rangle
        \]
        The differential on $X^{n, \w}_*$ is determined by
        \begin{equation}\label{defs17}
            \del T = \sum_{(S,e) \text{ an edge expansion of } T} S
        \end{equation}
        Define a composition map on $\circ_i : X^{n, \w}_* \otimes X^{m, \vv}_* \to X^{(n + m - 1), \w + \vv}_*$ by letting $T \circ_i S$ be the tree obtained from gluing the output leaf of $S$ to the $i$-th input leaf of $T$. 

        \begin{lemma}\label{defs16}
            \emph{[Lemma 4.8 from~\cite{DiagBible}]} $X^{*,*}_*$ is a chain complex with differential defined in~\eqref{defs17}, and one application of the differential drops the dimension by 1. Also, $\circ_i$ induces a chain map for each $i$.
        \end{lemma}

        We denote the chain map on $X^{*,*}_*$ induced by $\circ_i$ as
        \[
            \phi_{i,j,n;\vv, \w} : X^{j - i + 1, \vv}_* \otimes X^{n + i - j, \w}_* \to X^{n, \vv + \w}_*
        \]
        for $1 \leq i < j \leq n$, and $\vv, \w \in \Lambda$. We also define the \emph{corolla with $n$ inputs and weight $\w$} to be the $n$-input tree with 1 interior vertex, which is decorated with weight $\w$. We denote this corolla by $\Psi_n^{\w}$ for each $n, \w$.

        We next define the space of \emph{stably weighted bimodule trees}. $XB^{n,j,\w}_*$, to be the subcomplex of $X^{*,*}_*$ consisting of trees $T$ with $n + 1 + j$ inputs and total weight $\w$. 

        \begin{lemma}\label{defs18}
            $X^{n,j,\w}_*$ is a chain complex under the action of $\del$, and the map $\circ_{n + 1}: XB^{n,j,\w}_* \otimes XB^{m, k, \vv}_* \to XB^{(n+m), (j + k),(\vv + \w)}$ and the maps 
            \begin{align*}
                \circ_i &: XB^{n, j, \w}_* \otimes X^{m, \vv} \to XB^{n + m - 1, j, \vv + \w}_* \text{ for } i \leq n\\
                \circ_i &: XB^{n, j, \w}_* \otimes X^{m, \vv} \to XB^{n, j + m - 1, \vv + \w}_* \text{ for } i > n + 1
            \end{align*}
            are all chain maps.
        \end{lemma}

        \subsection{Weighted algebra diagonals}

        The goal of this section is to define the notion of a weighted algebra diagonal, which will be used to define tensor products of weighted $\A_{\infty}$-algebras. 

        First, we extend the dimension function to tensor products of trees additively, so $\dim (S\otimes T) = \dim S + \dim T$. For the constructions in this paper, we will always be working over polynomial ground rings. In order to construct reasonable bimodules and products, we make the following compatibility restriction:

         \begin{definition}\label{defs2}
            Compatible weight spaces and polynomial ground rings
        \end{definition}

        \noindent Let $R_0$ be some ring of coefficients. Consider polynomial weight spaces $R_1 = R_0[V_1, \ldots, V_m]$ and $R_2 = R_0[W_1, \ldots, W_n]$, and weight spaces $\Lambda_1 = \Z_{\geq 0} \langle \e_1, \ldots, \e_r\rangle$ and $\Lambda_2 = \Z_{\geq 0} \langle \f_1, \ldots, \f_s \rangle$. Then the pairs $(R_1, \Lambda_1)$ and $(R_2, \Lambda_2)$ are said to be compatible if and only if $r = n$ and $s = m$.

        When we need to specify the ground ring and lattice pair $(R, \Lambda)$ used to define the chain complex $X^{*,*}_*$ we will write it as $X^{*,*}_{*,(R,\Lambda)}$. 

        Consider a pair $(R_1, \Lambda_1), (R_2, \Lambda_2)$ of compatible weight spaces and polynomial ground rings. We view $X^{*,*}_{*, (R_1, \Lambda_1)} \otimes X^{*,*}_{*, (R_2, \Lambda_2)}$ as a module over $R = \F[V_1, \ldots, V_m, W_1, \ldots, W_n]$. For $S \otimes T \in X^{*,*}_{*, (R_1, \Lambda_1)} \otimes X^{*,*}_{*, (R_2, \Lambda_2)}$, we define
        \begin{align*}
            \wt_1(V_1^{s_1} \cdots V_m^{s_m} W_1^{t_1} \cdots W_n^{t_n} \cdot S \otimes T) &= \wt(S) +  \sum_{i = 1}^{m} s_i \f_i; \\
            \wt_2(V_1^{s_1} \cdots V_m^{s_m} W_1^{t_1} \cdots W_n^{t_n} \cdot S \otimes T) &= \wt(T) + \sum_{i = 1}^n t_i \e_i;
        \end{align*}

        Next, we extend $X^{*,*}_*$ to a complex $\tilde{X}^{*,*}_*$, which has the same generating set as $X^{*,*}_*,$ plus the 0-input tree $\top$ and the 1-input tree $\downarrow$. Both of these trees have dimension 0 according to~\eqref{defs12} (with the convention that $\top$ has $-1$ vertices and $\downarrow$ has 0 vertices). They interact with other trees in the following way:
        \begin{itemize}
            \item $\downarrow \circ_i T$ and $T \circ_i \downarrow$ are equal to $T$, for any $T \in \tilde{X}^{*,*}_*$;

            \item $\top \circ_i T \equiv 0$ for all $T \in X^{*,*}_*$ and $i$, and $T \circ_i \top$ is the tree obtained from $T$ by deleting the $i$-th input leaf;
        \end{itemize}

        Fix a pair $(R_1, \Lambda_1), (R_2, \Lambda_2)$ of compatible polynomial ground rings and weight spaces, notated as in Definition~\ref{defs2}. Write $R = R_0[V_1, \ldots, V_m, W_1, \ldots, W_n]$ as above, and write $\Lambda = \Lambda_1 + \Lambda_2$. We define a \emph{weighted seed $s = (s_1, \ldots, s_{m + n}$ for a weighted algebra diagonal} to be such that
        \begin{itemize}
             
             \item $s_i$ is a linear combination of  $\{W_i \Psi_0^{\e_i} \otimes \top\}_{j = 1}^{n}$, for each $1 \leq i \leq n$);
             
             \item $s_i$ is a linear combination of $\{V_i \top \otimes \Psi_0^{\f_i}\}_{j = 1}^m$, for each $n + 1 \leq i \leq m + n$;
        \end{itemize}

        \begin{definition}\label{defs20}
            Weighted algebra diagonal $\Gamma^{*,*}$ 
        \end{definition}

        \noindent A weighted algebra diagonal with seed $s$ is a map
        \begin{equation}\label{defs21}
            \Gamma^{n,\w}: X^{n, \w}_{*,(R,\Lambda)} \to \bigoplus_{{\tiny \begin{matrix}
                \w_1, \w_2 \leq \w \\
                \w_1 \in \Lambda_1 \\
                \w_2 \in \Lambda_2
            \end{matrix}}} \tilde{X}^{n, \w_1}_{*, (R_1, \Lambda_1)} \otimes \tilde{X}^{n, \w_2}_{*, (R_2, \Lambda_2)}
        \end{equation}
        satisfying the following conditions
        \begin{itemize}
            \item \textbf{Dimension preservation:} $\dim \Gamma^{n,\w}(T) = \dim T$ for each $n, \w,$ and $T \in  X^{n, \w}_{*,(R,\Lambda)}$;

            \item \textbf{Weight preservation:} For each $n, \w$, and $T \in  X^{n, \w}_{*,(R,\Lambda)}$,
            \[
                \wt_1(\Gamma^{n, \w}(T)) = \wt_2(\Gamma^{n, \w}(T)) = \wt(T) = \w;
            \]

            \item \textbf{Stacking:}
            \[
                \Gamma^{n, \vv + \w} \circ \phi_{i,j,n; \vv, \w} = \sum_{{\tiny \begin{matrix}
                    \vv_1 + \vv_2 = \vv \\
                    \w_1 + \w_2 = \w
                \end{matrix}}} (\phi_{i,j,n;\vv_1, \w_1} \otimes \phi_{i,j,n; \vv_2, \w_2}) \circ (\Gamma^{j - i + 1, \vv }\otimes \Gamma^{n + i - j, \w}.
            \]
            \item \textbf{Nondegeneracy:} $\Gamma^{*,*}$ is non-degenerate in the following sense:
            \begin{itemize}
                \item $X^{2,0}_*$ has a canonical generator, $\Psi_2^0$. We require that $\Gamma^{2,0}(\Psi_2^0) = \Psi_2^0 \otimes \Psi_2^0$;

                \item $\Gamma^{0,\e_i}(\Psi_0^{\e_i}) = s_i$ for each $1 \leq i \leq n$;

                \item $\Gamma^{0, \f_i}(\Psi_0^{\f_i}) = s_i$ for each $n + 1 \leq i \leq n + m$;

                \item For each $n, \w$, $\Gamma^{n, \w}$ has image in
                \[
                    \bigoplus_{{\tiny \begin{matrix}
                \w_1, \w_2 \leq \w \\
                \w_1 \in \Lambda_1 \\
                \w_2 \in \Lambda_2
            \end{matrix}}} (\tilde{X}^{n, \w_1}_{*, (R_1, \Lambda_1)} \otimes {X}^{n, \w_2}_{*, (R_2, \Lambda_2)}) \oplus ( {X}^{n, \w_1}_{*, (R_1, \Lambda_1)} \otimes \tilde{X}^{n, \w_2}_{*, (R_2, \Lambda_2)}),
                \]
                that is, for each $T$, each term of $\Gamma^{n,\w}(T)$ contains at most one factor which is one of $\top$ or $\downarrow$. 
            \end{itemize}
        \end{itemize}

        The proof that such diagonals exist is completely analgous to the proof from Section 6.2 of~\cite{DiagBible}.

       \subsection{Weighted bimodule diagonal primitives}\label{prim}
        In order to define the operations for a box tensor product, we will need the notion of a weighted bimodule diagonal primitive. This is inspired by the definition of a weighted module diagonal primitive -- see page 113 of~\cite{DiagBible}. In order to make this definition, we need a number of auxiliary definitions relating to concatenation of trees. 

        First, for $(T_1, S_1), (T_2, S_2) \in X^{*,*} \otimes XB^{*,*,*},$ 
        \[
            (T_2, S_2) \circ_{L,i} (T_1, S_1) = (T_2 \circ_i T_1, S_2 \circ_i S_1)
        \]
        and
        \[
            (T_2, S_2) \circ_{L} (T_1, S_1) = \sum_i (T_2, S_2) \circ_{L,i} (T_1, S_1).
        \]
        Likewise, for $(T_1, S_1), (T_2, S_2) \in X^{*,*} \otimes XB^{*,n, *},$ define
        \[
             (T_2, S_2) \circ_{R,i} (T_1, S_1) = (T_2 \circ_i T_1, S_2 \circ_{i + n + 1} S_1)
        \]
        and 
        \[
            (T_2, S_2) \circ_{R} (T_1, S_1) = \sum_i (T_2, S_2) \circ_{R,i} (T_1, S_1).
        \]
        Finally, if $T_0 \in X^{n,*}$ and $T_1, \ldots, T_n \in X^{*,*}$, let $T \circ (T_1, \ldots, T_n)$ be the tree obtained by gluing the input of $T_i$ to the $i$-th input of $T$, for each $i$
        Define $\RoJ: \underbrace{X^{*,*}_* \otimes \cdots \otimes X^{*,*}_*}_{m \text{ times}} \to X^{*,*}_*$ as
        \[
            \RoJ^{\w}(T_1, \ldots, T_m) = \Psi_m^{\w} \circ (T_1, \ldots, T_m)
        \]
        and
        \[
            \RoJ^* = \sum_{\w} \RoJ^{\w}.
        \]
        Let $\CeJ: \underbrace{XB^{*,*,*}_* \otimes \cdots \otimes XB^{*,*,*}_*}_{m \text{ times}} \to XB^{*,*,*}_*$ in the following way. For stably weighted trees $T_1, \ldots, T_m$ with $T_i \in \in X^{n_i, j_i, \w_i}_*$ for each $i$, let
        $\CeJ(T_1, \ldots, T_m)$ be the tree in $XB^{n, j, \w}$ obtained by gluing the output leaf of $T_{i - 1}$ to the $(n_i + 1)$st input leaf of $T_i$, for each $2 \leq i \leq m$, where we write 
        \begin{align*}
            n &= \sum_{i = 1}^m n_i, \\
            j &= \sum_{i = 1}^m j_i, \\
            \w &= \sum_{i = 1}^m \w_i.
        \end{align*} 
        We can then extend $\CeJ$ linearly to all of $XB^{*,*,*}_* \otimes \cdots \otimes XB^{*,*,*}_*.$

        Define $\CR^{\w}: (X^{*,*}_* \otimes XB^{*,*,*}_*)^{\otimes m} \to X^{*,*} \otimes XB^{*,*,*}_*$ as
        \[
            \CR((S_1, T_1), \ldots, (S_m, T_m)) = \RoJ^{\w}(S_1, \ldots S_m) \otimes \CeJ(T_1, \ldots, T_m),
        \]
        extending multilinearly. Let
        \[
            \CR^* = \sum_{\w} \CR^{\w}.
        \]

        We are now ready to give the definition of a weighted bimodule diagonal primitive.

        \begin{definition}\label{defs8}
            Weighted bimodule diagonal primitive $\p^{*,*,*}$ compatible with a weighted algebra diagonal $\Gamma^{*,*}$
        \end{definition}

       \noindent A weighted bimodule diagonal primitive is a collection of linear combinations of weighted trees
        \begin{equation}\label{defs9}
            \p^{n,j,\w} \in \bigoplus_{\w_1 + \w_2 \leq \w} X^{n, \w_1} \otimes XB^{n,j, \w_2}
        \end{equation}
        satisfying the following properties:
        \begin{itemize}
            \item \textbf{Compatibility with $\Gamma^{*,*}$}: For $n, j \geq 0$ with $(n,j) \notin \{(1,0), (0,1)\}$, 
            \begin{equation}\label{defs10}
                \del \p^{n,j,\w} = \sum_{{\tiny \begin{matrix}
                    \vv +\sum \w_i = \w \\
                    \sum n_i = n + k - 2\\
                    \sum j_i = j + k - 2
                \end{matrix}}} \CR^{\vv}(\p^{n_1, j_1, \w_1}, \ldots, \p^{n_k,j_k, \w_k}) + \sum_{{\tiny \begin{matrix}
                    \w_1 + w_2 = \w \\
                    n_1 + n_2 = n
                \end{matrix}}} \p^{n_2, j, \w_2} \circ_{L} \gamma^{n_1, \w_1} + \sum_{{\tiny \begin{matrix} \w_1 + \w_2 = \w \\ j_1 + j_2 = j \end{matrix}}} \p^{n, j_2, \w_2} \circ_R \Psi_{j_1}^{\w_1}
            \end{equation}

            \item \textbf{Base cases:}
            \begin{center}
              $\p^{1,0,0} =$  \hspace{10pt} \parbox{2cm}{\includegraphics[width = 1.5cm]{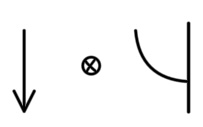}} 
              
                $\p^{0,1,0} = $  \parbox{2cm}{\includegraphics[width = 2cm]{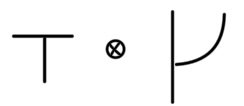}}
                
                $\p^{0,0,\e_i}= $  \parbox{2cm}{\includegraphics[width = 2cm]{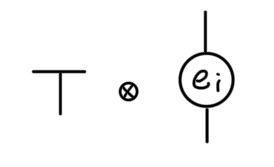}}
            \end{center}
        \end{itemize}

        Before we can prove the existence of weighted bimodule diagonal primitives, we need the following computational lemma. Throughout the following, we write $(S,T)$ as a shorthand for $S \otimes T$, so that $\dim(S,T) = \dim (S \otimes T)$.

        \begin{lemma}\label{defs27}
            \begin{enumerate}[label = (\alph*)]
                \item For any $(S_1, T_1), \ldots, (S_k, T_k) \in X^{*,*}_* \otimes XB^{*,*,*}_*$ with $\dim (S_i, T_i) = d_i$, for each $i$,tcul
                \begin{equation}\label{defs28}
                    \dim \CR^{\vv}((S_1, T_1), \ldots, (S_k, T_k)) = \sum_{i = 1}^k d_i + 2|\vv| + k - 2
                \end{equation}

                \item\label{defs32} $\dim ( (S, T) \circ_L (S', T')) = \dim (S,T) + \dim(S',T')$;

                \item $\dim( (S,T) \circ_R S') = \dim (S,T) + \dim S'$;

                \item\label{defs31} For any $(S_1, T_1), \ldots, (S_k, T_k) \in X^{*,*}_* \otimes XB^{*,*,*}_*$,
                \begin{align*}
                    \del \CR^{\vv}((S_1, T_1), \ldots, (S_k, T_k)) = &\sum_{i = 1}^k \CR^{\vv}((S_1, T_1), \ldots, \del (S_i, T_i), \ldots, (S_k, T_k)) \\
                    &+ \sum_{\tiny \begin{matrix}1 \leq i \leq i' \leq k \\ \vv_1 + \vv_2 = \vv \end{matrix}} \CR^{\vv_2}( (S_1, T_1), \ldots, \CR^{\vv_1}((S_i,T_i), \ldots, (S_{i'}, T_{i'})), \ldots, (S_k, T_k));
                \end{align*}

                \item\label{defs35} For $(S_1, T_1), \ldots, (S_k, T_k), (S,T) \in X^{*,*}_* \otimes X^{*,*,*}_*$, 
                \begin{equation}\label{defs30}
                    \CR^{\vv}((S_1, T_1), \ldots, (S_k, T_k)) \circ_L (S, T) = \sum_{i = 1}^k \CR^{\vv}(\cdots, (S_i, T_i) \circ_L (S, T),\cdots);
                \end{equation}

                \item\label{defs36} For $(S_1, T_1), \ldots, (S_k, T_k) \in X^{*,*}_* \otimes X^{*,*,*}_*$ and $S \in X^{*,*}_*$
                \begin{equation}\label{defs311}
                    \CR^{\vv}((S_1, T_1), \ldots, (S_k, T_k)) \circ_R S = \sum_{i = 1}^k \CR^{\vv}(\cdots, (S_i, T_i) \circ_R S,\cdots);
                \end{equation}

                \item For $(S,T), (S',T') \in X^{*,*}_* \otimes X^{*,*,*}_*$ and $S'' \in X^{*,*}_*$, we have
                \[
                    ((S,T) \circ_L (S',T')) \circ_R S'' = ( (S,T) \circ_R S'' ) \circ_L (S', T')
                \]

                \item $\del ((S,T) \circ_L (S',T')) = \del (S, T) \circ_L (S', T') + (S,T) \circ_L \del(S',T')$;

                \item\label{defs33} $\del ((S,T) \circ_R S') = \del (S, T) \circ_R S' + (S,T) \circ_R \del S'$;

                \item\label{defs34} $\del$ of the right hand side of~\eqref{defs10} is zero;
            \end{enumerate}
        \end{lemma}

        \begin{proof}
            For (a), note that if we have $S_1, \ldots, S_k \in X^{*,*}_*$ such that each $S_i$ has $n_i$ inputs, $v_i$ internal vertices, and total weight $\w_i$,
            \begin{align*}
                \dim \RoJ^{\vv}(S_1, \ldots, S_k) &= \sum n_i + 2 \sum |\w_i| + 2 |\vv| - \sum v_i - 2 \\
                &= \sum (n_i + 2 |\w_i| - v_i - 1) + 2 |\vv| + k - 2 \\
                &= \sum \dim S_i + 2|\vv| + k - 2.
            \end{align*}
            Likewise, for $T_1, \ldots, T_n \in X^{*,*,*}_*$ such that each $T_i \in X^{n_i,j_i,\w_i'}_*$ and has $v_i'$ internal vertices, we have
            \begin{align*}
                \dim \CeJ(T_1, \ldots, T_k) &= \sum n_i + \sum j_i + 1 + 2 \sum |\w_i'| - \sum v_i' - 1 \\
                &= \sum (n_i + j_i + 1 + 2 |\w_i| - v_i -1) \\
                &= \sum \dim T_i
            \end{align*}
            Thus
            \begin{align*}
                \dim \CR^{\vv}((S_1, T_1), \ldots, (S_k, T_k)) &= \sum ( \dim S_i + \dim T_i) + 2 |\vv| + k - 2 \\
                &= \sum d_i + 2 |\vv| + k - 2,
            \end{align*}
            as desired. Parts~\ref{defs32}-~\ref{defs33} are obvious from the definitions.

            For~\ref{defs34}, we work term by term. By~\ref{defs31},~\ref{defs35}, and~\ref{defs36}, and~\eqref{defs10},
            \begin{align*}
                 \del \CR^{\vv}(\p^{n_1, j_1,\w_1}, \ldots, \p^{n_k,j_k,\w_k}) &= \sum_{i = 1}^k \CR^{\vv}( \cdots, \del \p^{n_i, j_i, \w_i}, \cdots) \\
                & \qquad + \sum_{\tiny \begin{matrix} 1 \leq i \leq i' \leq k \\ \vv_1 + \vv_ = \vv \end{matrix}} \CR^{\vv_2} (\cdots, \CR^{\vv_1}(\p^{n_i, j_i, \w_i}, \ldots, \p^{n_{i'},j_{i'}, w_{i'}}), \cdots ) \\
                &=  \sum_{\tiny \begin{matrix} 
                1 \leq i \leq k \\
                n_i' + n_i'' = n_i + 1 \\ 
                \w_i' + \w_i'' = \w_i 
                \end{matrix}} \CR^{\vv}( \cdots , \p^{n_i'', j_i, \w_i''}, \cdots) \circ_L \gamma^{n_i', \w_i'}\\
                & \qquad + \sum_{\tiny \begin{matrix} 1\leq i \leq k \\ j_i' + j_i'' = j_i + 1 \\ \w_i' + \w_i'' = \w_i \end{matrix}} \CR^{\vv}( \cdots , \p^{n_i, j_i'', \w_i''}, \cdots) \circ_R \Psi_{j_i'}^{\w_i'}\\
                & \qquad + 2 \cdot \sum_{\tiny \begin{matrix} 1 \leq i \leq i' \leq k \\ \vv_1 + \vv_ = \vv \end{matrix}} \CR^{\vv_2} (\cdots, \CR^{\vv_1}(\p^{n_i, j_i, \w_i}, \ldots, \p^{n_{i'},j_{i'}, w_{i'}}), \cdots ) 
            \end{align*}
            Likewise
            \begin{align*}
                \del (\p^{n_2, j, \w_2} \circ_L \gamma^{n_1, \w_1}) &= (\del \p^{n_2, j, \w_2}) \circ_L \gamma^{n_1, \w_1} + \p^{n_2, j, \w_2} \circ_L (\del \gamma^{n_1, \w_1}) \\
                &= \sum_{\tiny \begin{matrix}
                    \sum n_i' = n_1 \\
                    \sum j_i = j \\
                    \sum \w_i' = \w_2 + \vv
                \end{matrix}} \CR^{\vv}(\p^{n_1', j_1, \w_1'}, \ldots, \p^{n_k', j_k, \w_k'}) \circ_L \gamma^{n_1, \w_1} \\
                & \qquad + \sum_{\tiny \begin{matrix}
                    n_1' + n_2' = n_2 \\
                    \w_1' + \w_2' = \w_2
                \end{matrix}} \p^{n_2', j, \w_2'} \circ_L \gamma^{n_1', \w_1'} \circ_L \gamma^{n_1, \w_1} \quad \text{\color{red} (*)} \\
                & \qquad + \sum_{\tiny \begin{matrix}
                    j_1 + j_2 = j \\
                    \w_1' + \w_2' = \w_2
                \end{matrix}} \p^{n_2, j_2, \w_2'} \circ_R \Psi_{j_1}^{\w_1'} \circ_{L} \gamma^{n_1, \w_1} \\
                & \qquad + \sum_{\tiny \begin{matrix}
                    n_1' + n_2' = n_1 \\
                    \w_1' + \w_2' = \w_1
                \end{matrix}} \p^{n_2, j, \w_2} \circ_L ( \gamma^{n_2', \w_2'} \circ \gamma^{n_1', \w_1'}) \quad \text{\color{red} (**)}
            \end{align*}
            and
            \begin{align*}
                \del ( \p^{n, j_2, \w_2} \circ_R \Psi_{j_1}^{\w_1}) &= (\del \p^{n, j_2, \w_2}) \circ_R \Psi_{j_1}^{\w_1} + \p^{n, j_2, \w_2} \circ_R (\del \Psi_{j_1}^{\w_1}) \\
                &= \sum_{\tiny \begin{matrix} 
                \sum n_i = n \\ 
                \sum j_i' = j_2 \\
                \sum \w_i' + \vv = \w_2 \\
                \end{matrix}} \CR^{\vv}(\p^{n_1, j_1', \w_1'}, \ldots, \p^{n_k, j_k', \w_k'}) \circ_R \Psi_{j_1}^{\w_1}\\ 
                & \qquad + \sum_{\tiny \begin{matrix}
                    n_1 + n_2 = n + 1 \\
                    \w_1' + \w_2' = \w_2
                \end{matrix}} \p^{n_2, j_2, \w_2'} \circ_L \gamma^{n_1, \w_1'} \circ_R \Psi_{j_1}^{\w_1} \\
                & \qquad + \cdot  \sum_{\tiny \begin{matrix}
                    j_1' + j_2' = j_2 \\
                    \w_1' + \w_2' = \w_2
                \end{matrix}} \p^{n, j_2', \w_2'} \circ_R \Psi_{j_1'}^{\w_1'} \circ_R \Psi_{j_1}^{\w_1} \quad \text{ \color{blue} (*)} \\
                & \qquad + \sum_{\tiny \begin{matrix}
                    j_1' + j_2' = j_1 \\
                    \w_1' + \w_2' = \w_1
                \end{matrix}} \p^{n, j_2, \w_2} \circ_R (\Psi_{j_2'}^{\w_2'}\circ \Psi_{j_1'}^{\w_1'}) \quad \text{ \color{blue} (**)}
            \end{align*}
            Then, taking the sum of each of these three over the appropriate indices, as on the right hand side of~\eqref{defs10}, the terms cancel in pairs. The only interesting points are the cancellations of the pairs of terms labeled {\color{red} (*) / (**)} and {\color{blue} (*) / (**)}, respectively. The red pair cancels because any term of the form $\cdots \circ_L \gamma^{**} \circ_L \gamma^{**}$ which is not of the form $\circ_L (\gamma^{**} \circ \gamma^{**})$ appears twice. Likewise blue pair cancels because any term of the form $\cdots \circ_R \Psi_*^*\circ_R \Psi_*^*$ which is not of the form $\cdots \circ_R(\Psi_*^* \circ \Psi_*^*)$ appears twice.
        \end{proof}

        \begin{proposition}\label{defs11}
            There exists a weighted bimodule diagonal primitive $\p^{*,*,*}$ satisfying the conditions above.
        \end{proposition}

        \begin{proof}
            The proof of this fact is by an acyclic models argument, using the dimension function from~\eqref{defs12}, above. 

            Note first that $\{\p^{n,0,0}\}_{n \geq 0}$ is just the weighted module diagonal primitive from Section 6.5 of~\cite{DiagBible}, so this part of the primitive exist by the arguments there.

            We will first prove the existence of an \emph{unweighted} bimodule diagonal primitive, that is, the terms $\{\p^{n,j,0}\}_{n, j}$. 

            We have as our base cases $\p^{1,0,0}$ and $\p^{0,1,0}$ from the non-degeneracy requirement above. Now,
            \begin{center}
                $\del \p^{2,0,0} = $ \vspace{.5cm} \parbox{6cm}{ \includegraphics[width = 6cm]{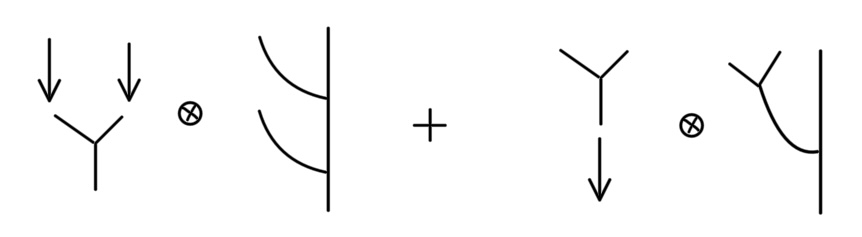}}
            \end{center}
            so that we can take
            \begin{center} 
                $\p^{2,0,0} = $ \parbox{2cm}{\includegraphics[width = 2cm]{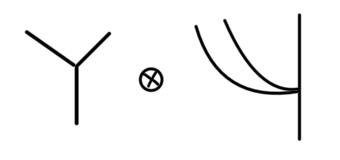}}
            \end{center}
            Likewise,
            \begin{center}
                $\del \p^{0,2,0} =$ \vspace{.5cm} \parbox{5cm}{\includegraphics[width = 5cm]{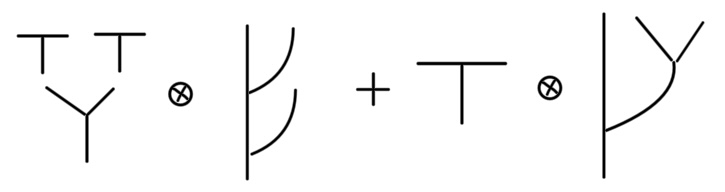}}
            \end{center}
            so we can take 
            \begin{center}
                $\p^{0,2,0} = $  \parbox{2cm}{\includegraphics[width = 2cm]{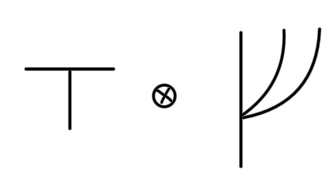}}
            \end{center}
            and
            \begin{center}
                $\del \p^{1,1,0}= $ \vspace{.5cm} \parbox{6cm}{\includegraphics[width = 6cm]{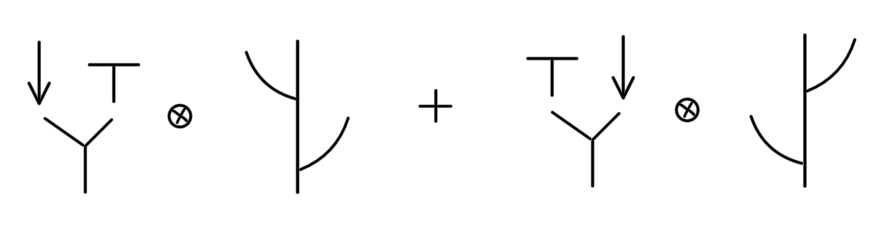}}
            \end{center}
            so we can take
            \begin{center}
                $\p^{1,1,0} = $ \parbox{2cm}{\includegraphics[width = 2cm]{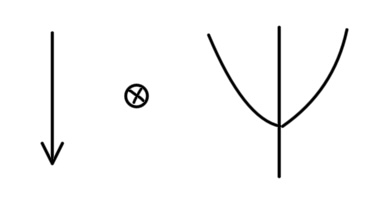}}
            \end{center}
            Notice that in each of the cases given above ($(n,j) = (1,0), (0,1), (2,0), (0, 2), (1,1)$), we have $\dim \p^{n,j,0} = n + j - 1$.
            
            Now, note that for each $n, j \geq 0$ (with $(n,j) \notin \{(1,0), (0,1)\}$, the terms on the right hand side of~\eqref{defs10} are all composed of $p^{n',j',\w'}$ such that at least one of $n' < n$, $j' < j$, and $\w' < \w$ is true. If we know  that 
            \begin{equation}\label{defs29}
                \dim \p^{n',j',\w'} = n' + j' + 2|\w' - 1 \text{ for each $(n',j',\w')$ with at least one of $n' < n$, $j'< j$, $\w' < \w$ true,}
            \end{equation}
           then it follows that the dimension of each of the terms on the right hand side of~\eqref{defs10} have dimension $n + j + 2|\w| - 2$. Indeed, given the hypothesis~\eqref{defs29},
            \begin{itemize}
                \item By Lemma~\ref{defs27}(a),
                \begin{align*}
                    \dim \CR^{\vv}( \p^{n_1, j_1,\w_1}, \ldots, \p^{n_k,j_k,\w_k}) &= \sum_{i = 1}^k \dim \p^{n_i, j_i, 0} + 2|\vv| k - 2 \\
                    &= \sum_{i = 1}^k (n_i + j_i + 2|\w_i| - 1) + 2|\vv|+ k - 2 \\
                    &= n + j + 2|\w| - 2
                \end{align*}
                since we are assuming $\sum n_i = n$, $\sum j_i = j$, and $\vv + \sum \w_i = \w$;

                \item By Lemma~\ref{defs27}(b),
                \begin{align*}
                    \dim (\p^{n_2,j, \w_2} \circ_L \gamma^{n_1, \w_1}) &= \dim \p^{n_2, j, \w_2} + \dim \gamma^{n_1, \w_1} \\
                    &= n_2 + j + 2|\w_2| - 1 + n_1 + 2|\w_1| - 2 \\
                    &= n + j + 2|\w| - 2
                \end{align*}
                since we are assuming $n_1 + n_2 = n + 1$ and $\w_1 + \w_2 = \w$;

                \item By Lemma~\ref{defs27}(c),
                \begin{align*}
                    \dim (\p^{n, j_2, \w_2} \circ_R \Psi_{j_1}^{\w_1}) &= \dim \p^{n, j_2, \w_2} + \dim \Psi_{j_1}^{\w_1} \\
                    &= n + j_2 + 2 |\w_2| - 1 + j_1 + 2|\w_1| - 2 \\
                    &= n + j + 2|\w| - 2,
                \end{align*}
                since we are assuming $j_1 + j_2 = j + 1$ and $\w_1 + \w_2 = \w$.
            \end{itemize}
            Now, the existence of an unweighted bimodule diagonal primitive follows from the base cases computed above (which are precisely those $\p^{n,j, 0}$ with $\dim \p^{n, j, 0} \in \{0, 1\}$) and from Lemma~\ref{defs27}~\ref{defs34} applied to the case of $\w = 0$. 

            We now construct aweighted bimodule diagonal primitive by induction on $\w$. By Lemma~\ref{defs27}~\ref{defs34} and the dimension counts above, it suffices to exhibit the inductive step for the $\p^{n,j,\w}$ with $\dim \p^{n,j,\w} = 2$ (one higher than the bottom dimension for weighted trees). According to the compatibility relation~\eqref{defs10},
            \begin{center}
                $\del \p^{1,0, \e_i} =$ \vspace{0.5cm}  \parbox{6cm}{\includegraphics[width = 6cm]{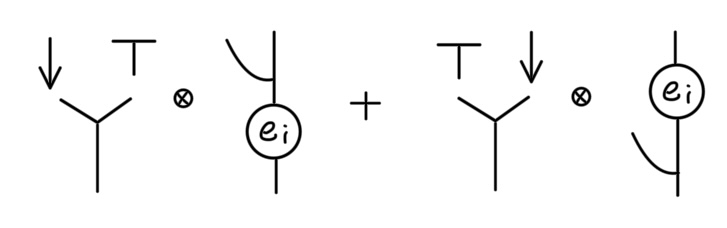}}
            \end{center}
            so we can take
            \begin{center}
                $\p^{1,0, \e_i} =$  \parbox{2cm}{\includegraphics[width = 2cm]{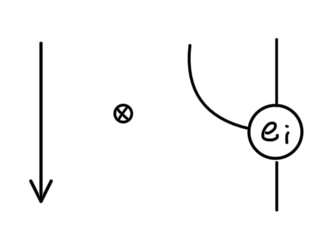}}
            \end{center}
            Likewise,
            \begin{center}
                $\del \p^{0,1, \e_i} =$ \vspace{0.5cm}  \parbox{6cm}{\includegraphics[width = 6cm]{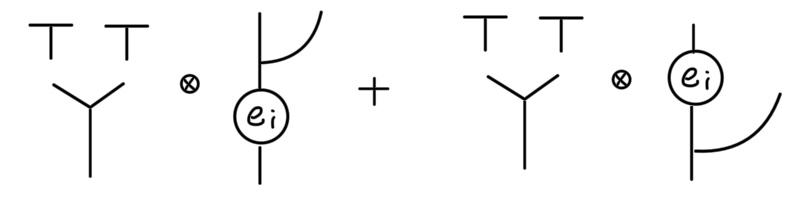}}
            \end{center}
            so we can take
            \begin{center}
                $\p^{0,1, \e_i} =$ \parbox{3cm}{\includegraphics[width = 3cm]{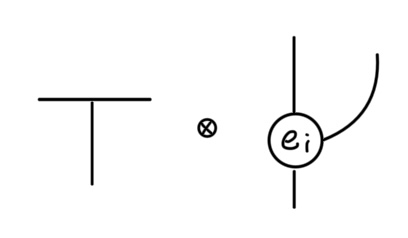}}
            \end{center}
            This completes the inductive step, and therefore the proof that there exists a weighted bimodule diagonal primitive compatible with the given weighted algebra diagonal $\Gamma^{*,*}$.
        \end{proof}
    
        \subsection{$\A_{\infty}$-algebras and -bimodules}\label{ainf}

        \begin{definition}\label{defs1}
            Weighted $\A_{\infty}$-algebras
        \end{definition}
        
        \noindent The data for a weighted $\A_{\infty}$-algebra is as follows:
        \begin{itemize}
            \item A ground ring $R$;

            \item A weight space $\Lambda$, consisting of $\Z_{\geq 0}$-linear combinations of some finite set of basic weight vectors $\{\e_i\}_{i = 1}^N$;

            For $\w = \sum_{i = 1}^N k_i e_i \in \Lambda$ define $|\w| = \sum_i k_i$; and for $\vv = \sum_{i = 1}^n \ell_i \e_i$, we say that $\vv \leq \w$ if and only if $k_i \leq \ell_i$ for each $i$.

            \item A chain complex $\A$ of $(R,R)$-bimodules;

            \item A collection of chain maps which are also $(R,R)$-bimodule homomorphims, $\mu_n^{\w}: \A^{\otimes n} \to \A$, for each $n \in \Z_{\geq 0}, \w \in \Lambda$, where $\A^{\otimes n}$ is equipped with the usual Koszul differential. We require that $\mu_1^0$ agrees with the differential on $\A$, and that the maps $\mu_n^{\w}$ satisfy the following \emph{$\A_{\infty}$-relations}: for each $n \in \Z_{\geq 0}, \w \in \Lambda$, and $a_1, \ldots, a_n \in \A$ 
            \begin{equation}\label{defs3}
                0 = \sum_{{\tiny \begin{matrix}0 \leq r \leq n \\ 0 \leq \vv\leq \w\end{matrix}}}\sum_{i = 1}^{n - r+1} \mu_{n - r + 1}^{\w - \vv} (a_1, \ldots, \mu_r^{\vv}(a_i, \ldots, a_{i + r - 1}), \ldots, a_n);
            \end{equation}
        \end{itemize}

        We can also write these $\A_{\infty}$-relations in a different form which will become useful in the discussion of trees, below. Define $\overline{D}^{\A, \w}: \TT^*\A \to \TT^*(A)$ by 
        \[
            \overline{D}^{\A, \w}(a_1 \otimes \cdots \otimes a_n) = \sum_{i + j \leq n + 1 } a_1 \otimes \dots a_{i - 1} \otimes \mu_j^{\w}(a_i,\ldots, a_{i + j - 1}) \otimes a_{i + j} \otimes \cdots \otimes a_n
        \]
        Then~\eqref{defs3} is equivalent to the stipulation that
        \begin{equation}\label{defs24}
            \sum_{\w_1 + \w_2 = \w}\overline{D}^{\A, \w_2} \circ \overline{D}^{\A, \w_1} \equiv 0
        \end{equation}
        for each $\w \in \Lambda$.

            In what follows, we will always be working over ground rings which contain a so-called \emph{ring of idempotents}. We therefore make the following definitions. 

        \begin{definition}\label{aux27}
            Ring of idempotents
        \end{definition}

        \noindent Let $R_0 = \F_2[ \I_1, \ldots, \I_N]$, where $\I_j$ are formal variables which interact int he following way:
         \[
            \I_i \I_j = \begin{cases}
                \I_i & i = j \\ 0 & \text{ otherwise}
            \end{cases}
         \]
         We call $R_0$ the ring of idempotents with $N$ generators, and make certain compatibility restrictions:

         \begin{definition}\label{aux28}
             Weighted $\A_{\infty}$-algebra with idempotents
         \end{definition}

         \noindent The data for an $\A_{\infty}$-algebra with idempotents is as follows:
         \begin{itemize}
             \item A weighted $\A_{\infty}$-algebra $\A$ with a ground ring $R$ which is a polynomial ring over $R_0$ (for some $N$), and some weight space $\Lambda$;

             \item We require that for each $a \in \A$, there exist unique $1 \leq i, j \leq N$ such that
             \[
                \I_i \cdot a = a \cdot \I_j = a
             \]
             and such that
             \[
                \I_{i'} \cdot a = a \cdot \I_{j'} = 0
             \]
             for each $i' \neq i$ and $j' \neq j$. We call $i$ the \emph{initial idempotent of $a$} and $j$ the \emph{final idempotent of $a$};

             \item We require that the $\A_{\infty}$-operations $\mu_n^{\w}$ are chain maps satisfying the $\A_{\infty}$-relations given above, and also the following additional requirement: $\mu_n^{\w}(a_1, \ldots, a_n) = 0$ unless the initial idempotent of $a_i$ is equal to the final idempotent of $a_{i - 1}$, for each $2 \leq i\leq n$;
         \end{itemize}

         \begin{remark}\label{aux30}
             \emph{Actually, the final condition follows from the existence of intial and final idempotents, because the $\mu_n^{\w}$ are required to be $R$-bimodule homomorphisms on the tensor product. But we state it separately for emphasis.}
         \end{remark}

        We can (and often will) use the operad of trees to induce an $\A_{\infty}$-algebra structure on a chain complex $\A$. More precisely, let $\A$ be a chain complex of $R$-modules, and suppose we are given maps $\mu_n^{\w}: \A^{\otimes n} \to \A$ for each $n, \w$. Consider the family of maps
        \begin{equation}\label{defs23}
            \mu: X^{n,\w}_* \to \mathrm{Mor}(\A^{\otimes n}, \A), \: n \in \Z_{\geq 0} \text{ and } \w \in \Lambda
        \end{equation}
        defined on the generating trees $T \in X^{n, \w}$ by replacing each vertex of $T$ with weight $\w$ and valence $n + 1$ with $\mu_n^{\w}$, and then composing according to the edges of $T$. Then:

        \begin{lemma}\label{defs22}
            \emph{(Lemma 4.12 from~\cite{DiagBible})}$\A$ is an $\A_{\infty}$-algebra with operations $\{\mu_n^{\w}\}_{n, \w}$, if and only if the maps $\mu$ from~\eqref{defs23} are chain maps for each $n, \w$.
        \end{lemma}
        
        This means that we can express the $\A_{\infty}$-operations and -relations of an $\A_{\infty}$-algebra $\A$ with operations $\{\mu_n^{\w}\}$ in terms of trees. Namely, we write $\mu_n^{\w}$ as $\Psi_n^{\w}$ for each $n$, and write the $\A_{\infty}$-relations~\eqref{defs24} as
        \begin{center} $\sum_{\w_1 + \w_2 = \w}$
        \begin{tikzcd}
            \arrow[d, Rightarrow]  \\
            \overline{D}^{\A, \w_1} \arrow[d, Rightarrow]\\ 
            \overline{D}^{\A, \w_2} \arrow[d, Rightarrow] \\
            \:
        \end{tikzcd} $=0$
        \end{center}
        for each $\w \in \Lambda$.

        Lemma~\ref{defs22} also implies that given $\A_{\infty}$-algebras $\A$ and $\B$, a weighted algebra diagonal $\Gamma^{*,*}$ induces an $\A_{\infty}$-algebra structure on the $(R_1, R_2)$-bimodule $\A \otimes \B$, in the following way. Let all notation be as in Definition~\ref{defs20}. For each $n, \w$, define $\gamma^{n, \w} = \Gamma^{n, \w}(\Psi_n^{\w}).$ For each $n, \w$, we have 
        \[  
            \gamma^{n, \w} \in \bigoplus_{\w_1, \w_2 \leq \w} \tilde{X}^{n, \w_1}_* \otimes \tilde{X}^{n, \w_2}_*
        \]
        Therefore, by using the maps $\mu$ from~\eqref{defs23}, for $\A$ and $\B$ respectively, each $\gamma^{n,\w}$ determines an operation $(\A \otimes \B)^{\otimes n} \to \A \otimes \B$. For each $n, \w$, denote this operation by $\mu_n^{\w, \A \otimes \B}.$ Then since $\Gamma^{*,*}$ is a chain map (by the stacking condition), it follows from Lemma~\ref{defs22} that this process induces a valid $\A_{\infty}$-structure on $\A \otimes \B$.

         We will need the following notion of boundedness to define box tensor products in Section~\ref{boxy}. We say that a weighted $\A_{\infty}$-algebra with operations $\{\mu_n^{\w}\}$ is \emph{bonsai} if there is a single $N \in \Z$ such that if $T$ is a stably weighted tree with $\dim T > N$, then $\mu(T) = 0$ (where $\mu$ is as in~\eqref{defs23}).

         Finally, we define the notion of maps between weighted $\A_{\infty}$-algebras:

         \begin{definition}\label{aux2}
             Homomorphism of weighted $\A_{\infty}$-algebras
         \end{definition}

         \noindent The data for such a homomorphism is as follows
         \begin{itemize}
             \item Weighted $\A_{\infty}$ algebras $\A, \B$ over a single ground ring $R$ and weight space $\Lambda$;

             \item Chain maps $\varphi_n^{\w}: \A^{\otimes n} \to \B$, for each $n \in \Z_{\geq 0}$ and $\w \in \Lambda$. Define
             \begin{align*}
                \varphi^{0, \w} &= 0 \text{ for each } \w \\
                \varphi^{1, \w} &= \sum_{j \geq 0} \varphi_{j}^{\w} \\
                \varphi^{k, \w} &= \sum_{\w_1 + \w_2 = \w}(\id_{T^*(\B)} \otimes \varphi^{1, \w_2}) \circ (\varphi^{(k - 1), \w_1} \otimes \id_{T^*(\A)}) \circ \Delta \text{ for each } k \geq 2,
            \end{align*}
            where $\Delta: T^*(\A) \to T^*(\A) \otimes T^*(\A)$ is the standard comultiplication map; so $\varphi^{k, \w}: T^*(\A) \to \B^{ \otimes k}$ for each $k$. Define
            \[
                \varphi^{\w} = \sum_{k \geq 0} \varphi^{k, \w}
            \]
            We require that these maps satisfy the following $\A_{\infty}$-relations:
            \begin{equation}\label{aux3}
               \sum_{\w_1 + \w_2 = \w} \varphi^{\w_2} \circ \overline{D}^{\A, \w_1} + \overline{D}^{\B, \w_2}  \circ \varphi^{\w_1} = 0
            \end{equation}
         \end{itemize} 

         Additionally, we note that any homomorphism between weighted $\A_{\infty}$-algebras with idempotents satisfies the following compatibility condition.

         \begin{lemma}\label{aux29}
             Let $\A$ and $\B$ be weighted $\A_{\infty}$-algebras over a weight space $\Lambda$ and ground ring $R$, which is a polynomial ring over $R_0$ (for some $N$). Assume $\A$ and $\B$ both satisfy the conditions of Definition~\ref{aux28}, and let $\varphi: \A \to \B$ be a weighted $\A_{\infty}$-algebra homomorphism. Then for any $a_1, \ldots, a_n \in \A$ and $\w \in \Lambda$,
             \[
                \varphi_n^{\w}(a_1, \ldots, a_n) = 0
             \]
             unless the initial idempotent of $a_i$ is equal to the final idempotent of $a_{i - 1}$ for each $2 \leq i \leq n$.
         \end{lemma}

        This is clear from the definitions, since $\varphi_n^{\w}$ is a ring homomorphism on the tensor product for each $n \in \Z_{\geq 0}$.
        
         We now define the first type of $\A_{\infty}$-bimodule which appears in Theorem~\ref{Koszul}

        \begin{definition}\label{defs41}
            Weighted AA-bimodules
        \end{definition}

        \noindent The data for an $AA$-bimodule $\:_{\A} Y_{\B}$ is as follows:
        \begin{itemize}
            \item Weighted $\A_{\infty}$-algebras $\A$ and $\B$ over polynomial ground rings $R_1, R_2$, respectively, and weight spaces $\Lambda_1, \Lambda_2$, respectively, such that $(R_1, \Lambda_1)$, $(R_2, \Lambda_2)$ are compatible in the sense of Definition~\ref{defs2}, above;

            \item An $(R_1,R_2)$-bimodule $Y$;

            \item $(R_1, R_2)$-bimodule maps $m_{n|j}^{\w}: \A^{\otimes_{R_1} n} \otimes_{R_1} Y \otimes_{R_2} \B^{\otimes_{R_2} j} \to Y$, where $n, j \in \Z_{\geq 0}$ and $\w \in \Lambda_1 + \Lambda_2$, which satisfy the following $\A_{\infty}$-relations: for $n, j \in \Z_{\geq 0}, a_1, \ldots, a_n \in \A, b_1, \ldots, b_j \in \B, \vv \in \Lambda_1,$ and $\w \in \Lambda_2$
            \begin{align*}
                0 = \sum_{{\tiny\begin{matrix}0 \leq r \leq n \\ \vv' \leq \vv \end{matrix}}}\: \sum_{ 0 \leq i \leq n - r + 1} &m_{n - r + 1| j}^{\vv + \w - \vv'}(a_n, \ldots, \mu_r^{\vv'}(a_{i + r - 1}, \ldots, a_{i}), \ldots, a_1, \y, b_1, \ldots, b_n) \\
                & + \sum_{{\tiny\begin{matrix}0 \leq s \leq j\\ \w' \leq \w \end{matrix}}}\: \sum_{ 0 \leq i \leq j - s + 1} m_{n | j- s + 1}^{\vv + \w - \w'}(a_n, \ldots, a_1, \y, b_1, \ldots, \mu_s^{\w'}(b_{i}, \ldots, b_{i+s-1}) \ldots, b_n) \\
                & + \sum_{{\tiny \begin{matrix} 0 \leq r \leq n \\ 0 \leq s \leq j \\ \vv' \leq \vv \\ \w' \leq \w \end{matrix}}} m_{(n - r)|(j - s)}^{\vv + \w - \vv' - \w'}(a_n, \ldots, m_{r|s}^{\vv' + \w'}(a_r, \ldots a_1, \y, b_1 \ldots, b_s), \ldots, b_j)
            \end{align*}
        \end{itemize}

        When we are working over weighted $\A_{\infty}$-algebras with idempotents, we make the following additional compatibility restrictions:

        \begin{definition}\label{aux31}
            Weighted AA-bimodules with idempotents
        \end{definition}

        \noindent The data for such a structure is as follows
        \begin{itemize}
            \item A ring of idempotents $R_0$, for some fixed $N$;
            
            \item Weighted $\A_{\infty}$-algebras over weight spaces $\Lambda_1, \Lambda_2$, and ground rings $R_1, R_2$ that are polynomial rings over $R_0$. We still require $(R_1, \Lambda_1)$, $(R_2, \Lambda_2)$ to be compatible in the sense of Definition~\ref{defs2}; 

            \item A weighted AA-bimodule $\:_{\A} Y_{\B}$ over $\A$ and $\B$, satisfying the $\A_{\infty}$-relations of Definition~\ref{defs41}.

            \item We require that each for each $\y \in Y$ there exist unique $1 \leq i, j \leq N$ such that
             \[
                \I_i \cdot \y = \y \cdot \I_j = \y
             \]
             and such that
             \[
                \I_{i'} \cdot \y = \y \cdot \I_{j'} = 0
             \]
             for each $i' \neq i$ and $j' \neq j$. We call $i$ the \emph{initial idempotent of $\y$} and $j$ the \emph{final idempotent of $\y$};

             \item We require that for each $\y \in Y$, $a_1, \ldots, a_n \in \A$, $b_1, \ldots, b_j \in \B$, and $\w \in \Lambda$, 
             \[
                m_{n|j}^{\w}(a_n,\ldots, a_1, \y, b_1, \ldots, b_j) = 0
             \]
             unless
             \begin{itemize}
                 \item The initial idempotent of $a_i$ is equal to the final idempotent of $a_{i + 1}$ for each $1 \leq i \leq n - 1$; 
                 \item The initial idempotent of $\y$ is equal to the final idempotent of $a_1$;
                 \item The initial idempotent of $b_1$ is equal to the final idempotent of $\y$;
                 \item The initial idempotent of $b_i$ is equal to the final idempotent of $b_{i -1}$, for each $1 \leq 2 \leq j$;
             \end{itemize}
        \end{itemize}

        \begin{remark}\label{aux31}
            \emph{Again, the final condition of Definition~\ref{aux30} follows from the existence of initial and final idempotents, but we state it in the definition for emphasis.}
        \end{remark}

        We can write the $\A_{\infty}$-relations from Definition~\ref{defs41} in a different way, which will be useful when we discuss them in terms of trees. Define $m^{\w}: \TT^*\A \otimes Y \otimes \TT^*(B) \to Y$ as
        \[
            m^{\w}|_{\A^{\otimes n}\otimes Y \otimes \B^{\otimes j}} = m_{n|j}^{\w},
        \]
        for each $n,j$.

        We can express AA-bimodule operations in terms of stably weighted bimodule trees $T \in XB^{*,*,*}_*$ in the following way. Let $Y$ be an AA-bimodule over $\A, \B$ as defined above. Note that each bimodule tree $T \in X^{n,j, \w}_*$ has a distinguished input leaf with $n$ input leaves to the left and $j$ to the right. There is a unique path $S$ from this $(n + 1)$st input leaf to the output leaf in $T$, which we can view as a sub-tree of $T$. We can orient $S$ from its input vertex to its output vertex, so that each vertex $v$ on $S$ has one leaf on $S$ which is closer to the input vertex of $S$, which we call the \emph{input leaf at $v$}, and one leaf on $S$ which is closer to the output vertex, which we call the \emph{output leaf at $v$}. For each interior vertex v on $S$, each leaf abutting $v$ is either to the left of the input and output leaves of $S$ at $v$, or to the right. For interior vertex $v$ in $S$, let $n(v)$ denote the number of leaves abutting $v$ to the left of the input / output leaves of $S$ at $v$, and let $j(v)$ denote the number of leaves abutting $v$ to the right. For each vertex of $T$, let $\w(v)$ denote the weight of $v$. Since $S$ bisects the tree $T$, so that every vertex of $T$ is either on $S$, to the left, or to the right. Replace vertices with $\A_{\infty}$-operations in the following way:
        \begin{itemize}
            \item Replace each interior $(n + 1)$-valent vertex $v$ to the left of $S$ with the operation $\mu_n^{\w(v), \A}$;
            
            \item Replace each interior $(n + 1)$-valent vertex $v$ to the right of $S$ with the operation $\mu_n^{\w(v), \B}$;

            \item Replace $v$ with the bimodule operation $m_{n(v),j(v)}^{\w(v)}$
        \end{itemize}
        Then compose all of these operations according to the leaves of $T$. This determines a collection of maps:
        \begin{equation}\label{defs25}
            m: XB^{n,j,\w}_* \to \mathrm{Mor}(\A^{\otimes n} \otimes Y \otimes \B^{\otimes j}, Y) \text{ for each } n, j \in \Z_{\geq 0}, \w \in \Lambda.
        \end{equation}
        More specifically, the procedure above defined $m$ on the generators of each $XB^{*,*,*}$, and $m$ can then be extended linearly to the full space.

        As in the case of $\A_{\infty}$-algebras, the following lemma is clear:

        \begin{lemma}\label{defs26}
            An $(R_1, R_2)$-bimodule $Y$ with operations $m_{n,j}^{\w}$ is an $\A_{\infty}$-algebra if and only if the maps $m$ are chain maps for each $n, j, \w$.
         \end{lemma}

        In terms of trees, the $\A_{\infty}$-relations from Definition~\ref{defs41} appear (for each $\w \in \Lambda$) as 
        \begin{center}
        \begin{tikzcd}[column sep = tiny, row sep = small]
            \arrow[dr, Rightarrow, no head] & & \arrow[dd, no head] & \arrow[ddl, Rightarrow, no head] \\
            & \overline{D}^{\A, \w_1} \arrow[dr, Rightarrow, no head] & & \\
            & & m^{\w_2} \arrow[d] & \\
            & & \: &
        \end{tikzcd} $+$
        \begin{tikzcd}[column sep = tiny, row sep = small]
            \arrow[ddr, Rightarrow, no head] & \arrow[dd, no head] & & \arrow[dl, Rightarrow, no head] \\
            & & \overline{D}^{\B, \w_1} \arrow[dl, Rightarrow, no head] & \\
            & m^{\w_2} \arrow[d] & & \\
            & \: & &
        \end{tikzcd} $+$
        \begin{tikzcd}[column sep = tiny, row sep = small]
            \arrow[dr, Rightarrow, no head] & & \arrow[dd, no head] & & \arrow[dl, Rightarrow, no head] \\
            & \Delta \arrow[dr, Rightarrow, no head] \arrow[ddr, bend right, Rightarrow, no head] & & \Delta \arrow[dl, Rightarrow, no head] \arrow[ddl, bend left, Rightarrow, no head] & \\
            & & m^{\w_1} \arrow[d, no head] & & \\
            & & m^{\w_2} \arrow[d] & & \\
            & & \: & & 
        \end{tikzcd} $= 0$
        \end{center}

        \begin{definition}\label{defs4}
            DD-bimodules over weighted $\A_{\infty}$-algebras
        \end{definition}

        \noindent The data for a DD-bimodule over a pair of weighted $\A_{\infty}$-algebras is as follows:

        \begin{itemize}
            \item Weighted $\A_{\infty}$-algebras $\A$ and $\B$ over polynomial ground rings $R_1, R_2$, respectively, and weight spaces $\Lambda_1, \Lambda_2$, respectively, such that $(R_1, \Lambda_1)$, $(R_2, \Lambda_2)$ are compatible in the sense of Definition~\ref{defs2}, above;

            \item A weighted algebra diagonal $\Gamma^{*,*}$;

            \item An $(R_1,R_2)$-bimodule $X$;

            \item An $(R_1, R_2)$-bimodule operation $\delta^1: X \to \A \otimes_{R_1} X \otimes_{R_2} B$. This gives rise to operations $\delta^n: X \to \A^{\otimes n} \otimes X \otimes \B^{\otimes n}$, which satisfy the following $\A_{\infty}$-relations
            \begin{equation}\label{defs5}
                \sum_{{\tiny \begin{matrix} n \in \Z_{\geq 0} \\ \vv \in \Lambda_1 \\ \w \in \Lambda_2 \end{matrix}}} (\mu_n^{(\vv + \w), \A \otimes \B} \otimes \id_X) \circ \delta^n = 0,
            \end{equation}
            where $\mu_*^{*, \A \otimes \B}$ denote the weighted $\A_{\infty}$-operations defined on $\A \otimes \B$ by the algebra diagonal $\Gamma^{*,*}$.
        \end{itemize}

        The sum in~\eqref{defs5} is finite under certain conditions:
        \begin{lemma}\label{aux35}
            If $\A$ and $\B$ are both bonsai (so that $\A \otimes \B$ is) \emph{or} if $\delta^n \equiv 0$ for all sufficiently large $n$, then the sum in~\eqref{defs5} is finite. 
        \end{lemma}
        This is a remark made on page 129 of~\cite{DiagBible}. 

        When we are working over weighted $\A_{\infty}$-algebras with idempotents, we make the following additional compatibility relations:

        \begin{definition}\label{aux31}
            DD-bimodules with idempotents
        \end{definition}

        \noindent The data for such a structure is as follows
        \begin{itemize}
            \item A ring of idempotents $R_0$, for some fixed $N$;
            
            \item Weighted $\A_{\infty}$-algebras $\A$ and $\B$ over weight spaces $\Lambda_1, \Lambda_2$, and ground rings $R_1, R_2$ that are polynomial rings over $R_0$. We still require $(R_1, \Lambda_1)$, $(R_2, \Lambda_2)$ to be compatible in the sense of Definition~\ref{defs2}; 

            \item A DD-bimodule $\:_{\A} X_{\B}$ over $\A$ and $\B$, satisfying the $\A_{\infty}$-relations of Definition~\ref{defs4}.

            \item We require that each for each $\x \in X$ there exist unique $1 \leq i, j \leq N$ such that
             \[
                \I_i \cdot \x = \x \cdot \I_j = \x
             \]
             and such that
             \[
                \I_{i'} \cdot \x = \x \cdot \I_{j'} = 0
             \]
             for each $i' \neq i$ and $j' \neq j$. We call $i$ the \emph{initial idempotent of $\x$} and $j$ the \emph{final idempotent of $\x$};

             \item We require that for each $\x\in X$, if we write
             \[
                \delta^1 \x = \sum_{r = 1}^n a_r \otimes \x_r \otimes b_r,
             \]
             then for each $r$ the initial idempotent of $\x$ is equal to the final idempotent of $a_r$, and the final idempotent of $\x$ is equal to the intial idempotent of $b_r$;
        \end{itemize}

        We can also express the $\A_{\infty}$-operations and -relations on a $DD$-bimodule in terms of trees. We write $\delta^1$ as
        \begin{center}
        \begin{tikzcd}[column sep = tiny]
             & \arrow[d] & \\
             & \delta^1 \arrow[d] \arrow[dl] \arrow[dr] & \\
             \: & \: & \:
        \end{tikzcd}
        \end{center}
        so that the $\A_{\infty}$-relations appear as
        \begin{center} $\sum_{n, \w}$
        \begin{tikzcd}[column sep = tiny]
            & & \arrow[d, no head] & & \\
            & & \delta^{n} \arrow[dd] \arrow[dl, Rightarrow] \arrow[dr, Rightarrow] & & \\
            & \arrow[dl, bend right] \mu(\gamma^{n, \w}) & & \mu(\gamma^{n, \w}) \arrow[dr, bend left] & \\
            \: & & \: & & \:
        \end{tikzcd} $= 0$
        \end{center}

        \begin{definition}\label{defs6}
            Weighted DA-bimodule $\:^{\A} M_{\B}$ over weighted $\A_{\infty}$-algebras $\A$ and $\B$
        \end{definition}

       \noindent The data for a weighted DA-bimodule is as follows:
        \begin{itemize}
            \item Weighted $\A_{\infty}$-algebras $\A$ and $\B$ over polynomial ground rings $R_1, R_2$, respectively, and a single weight spaces $\Lambda$; 

            \item An $(R_1, R_2)$-bimodule $M$;

            \item $(R_1, R_2)$-bimodule operations $\delta_{j + 1}^{1,\w}: M \otimes_{R_1} \A^{\otimes j} \to \B \otimes_{R_2} M$, for $j \in \Z_{\geq 0}, \w \in \Lambda$. Define
            \begin{align*}
                \delta^{0, \w} &= \id_M \text{ for each } \w \\
                \delta^{1, \w} &= \sum_{j \geq 0} \delta_{1 + j}^{1, \w} \\
                \delta^{k, \w} &= \sum_{\w_1 + \w_2 = \w}(\id_{T^*(\A)} \otimes \delta^{1, \w_2}) \circ (\delta^{(k - 1), \w_1} \otimes \id_{T^*(\B)}) \circ (\id_M \otimes \Delta) \text{ for each } k \geq 2,
            \end{align*}
            where $\Delta: T^*(\B) \to T^*(\B) \otimes T^*(\B)$ is the standard comultiplication map; so $\delta^{k, \w}: M \otimes T^*(\B) \to \A^{ \otimes k} \otimes M$ for each $k$. Define
            \[
                \delta^{M, \w} = \sum_{k \geq 0} \delta^{k, \w}
            \]
            We require that these maps satisfy the following $\A_{\infty}$-relations:
            \begin{equation}\label{defs7}
               \sum_{\w_1 + \w_2 = \w} \delta^{M, \w_2} \circ (\id_M \otimes \overline{D}^{\B, \w_1})  + (\overline{D}^{\A, \w_2} \otimes \id_M) \circ \delta^{M, \w_1} = 0
            \end{equation}  
        \end{itemize}

        When we are working over weighted $\A_{\infty}$-algebras with idempotents, we make the following additional compatibility restrictions, analogous to those for the other types of $\A_{\infty}$-bimodules, above.

        \begin{definition}\label{aux32}
            Weighted DA bimodule with idempotents
        \end{definition}

        \noindent The data for such a structure is as follows
        \begin{itemize}
            \item A ring of idempotents $R_0$, for some fixed $N$;
            
            \item Weighted $\A_{\infty}$-algebras over weight spaces $\Lambda_1, \Lambda_2$, and ground rings $R_1, R_2$ that are polynomial rings over $R_0$. We still require $(R_1, \Lambda_1)$, $(R_2, \Lambda_2)$ to be compatible in the sense of Definition~\ref{defs2}; 

            \item A weighted DA-bimodule $\:^{\A} M_{\B}$ over $\A$ and $\B$, satisfying the $\A_{\infty}$-relations of Definition~\ref{defs6}.

            \item We require that each for each $\m \in M$ there exist unique $1 \leq i \leq N$ such that
             \[
                \I_i \cdot \m = \m \cdot \I_i = m
             \]
             and such that
             \[
                \I_{i'} \cdot \m = \m \cdot \I_{i'} = 0
             \]
             for each $i' \neq i$. We say that \emph{$\m$ is in the $i$-th idempotent}, or that $\m$ has both initial and final idempotent $i$.

             \item We require that for each $\m \in Y$, $b_1, \ldots, b_j \in \B$, and $\w \in \Lambda$, if we write
             \begin{equation}\label{aux33}
                \delta_{1 + j}^{1, \w}(\m, b_1, \ldots, b_j) = 0
             \end{equation}
             unless the initial idempotent of $b_1$ is equal to the final idempotent of $\m$ and the initial idempotent of $b_i$ is equal to the final idempotent of $b_{i -1}$, for each $1 \leq 2 \leq j$. If the expression on the left hand side of~\eqref{aux33} is non-zero, say
             \[
                \delta_{1 + j}^{1, \w}(\m, b_1, \ldots, b_j) = \sum_{i = 1}^k a_{i,r_i} \otimes \cdots \otimes a_{i, 1} \otimes \m_i
             \]
             then for each $1 \leq i\leq k$:
             \begin{itemize}
                 \item The initial idempotent of $\m_i$ equals the final idempotent of $a_{i,1}$;

                 \item For each $1 \leq \ell \leq n - 1$, the initial idempotent of $a_{i,\ell}$ is equal to the final idempotent of $a_{i, \ell + 1}$;
             \end{itemize}
        \end{itemize}

        We can also express DA-bimodule operations in terms of trees, that is:
        \begin{center}
        \begin{tikzcd}[column sep = tiny]
            & \arrow[d] & \arrow[dl, Rightarrow] \\
            & \delta^{1, \w} \arrow[d] \arrow[dl] & \\
            \: & \: &
        \end{tikzcd} \hspace{.5cm} and
        \begin{tikzcd}[column sep = tiny]
            & \arrow[d] & \arrow[dl, Rightarrow] \\
            & \delta^{k, \w} \arrow[d] \arrow[dl, Rightarrow] & \\
            \: & \: &
        \end{tikzcd} \hspace{.5cm} and
        \begin{tikzcd}[column sep = tiny]
            & \arrow[d] & \arrow[dl, Rightarrow] \\
            & \delta^{M, \w} \arrow[d] \arrow[dl, Rightarrow] & \\
            \: & \: &
        \end{tikzcd}
        \end{center}
        so that the $\A_{\infty}$-relation from Definition~\ref{defs6} is
        \begin{center}$\sum_{\w_1 + \w_2 = \w}$
        \begin{tikzcd}[column sep = tiny]
           & \arrow[dd] & & \arrow[dl, Rightarrow] \\
           & & \overline{D}^{\B, \w_1} \arrow[dl, Rightarrow] & \\
           & \delta^{M, \w_2} \arrow[dl, Rightarrow] \arrow[d] & & \\
           \: & \: & & 
        \end{tikzcd} $+$
        \begin{tikzcd}[column sep = tiny]
            & & \arrow[d] & \arrow[dl, Rightarrow] \\
            & & \delta^{M, \w_1} \arrow[dl, Rightarrow] \arrow[dd] & \\
            & \overline{D}^{\A, \w_2} \arrow[dl, Rightarrow] & &\\
            \: & & \: &
        \end{tikzcd} $= 0$
        \end{center}

        In particular, we define the identity bimodule:

        \begin{definition}\label{aux1}
            Identity DA-bimodule over $\A$, $\:^{\A} \id_{\A}$
        \end{definition}

        \noindent For a given weighted $\A_{\infty}$-algebra $\A$ with , define $\:^{\A} \id_{\A} = \A$ as an $(R,R)$-bimodule. Define 
        \[
            \delta_{j + 1}^{1, \w} = \begin{cases}
                \id_{\A} & \text{ if } \w = 0, j = 1; \\
                0 & \text{ otherwise}
            \end{cases}
        \]

        Note that the operations on $\:^{\A} \id_{\A}$ clearly satisfy the requirements of Definition~\ref{aux32}. Next, note that any weighted $\A_{\infty}$-algebra homomorphism gives rise to a weighted DA-bimodule:

        \begin{definition}\label{aux4}
            Weighted DA-bimodule $\:^{\B} [\varphi]_{\A}$
        \end{definition}

        \noindent Let $\A, \B$ be weighted $\A_{\infty}$-algebras with a single ground ring $R$ and a single weight-space $\Lambda$. Let $\varphi: \A \to \B$ be a weighted $\A_{\infty}$-algebra homomorphism. Let $M$ be an $(R,R)$-bimodule which has a single generator $\x$ -- i.e. one which is isomorphic to $R$ as an $(R,R)$-bimodule. Define
        \begin{equation}\label{aux6}
            \delta_{j + 1}^{1, \w}(\x, a_1, \ldots, a_j) = \varphi_j^{\w}(a_1, \ldots, a_j) \otimes x
        \end{equation}
        for each $j \in \Z_{\geq 0}, \w \in \Lambda$, and $a_1, \ldots, a_j \in \A$. By inspection of the $\A_{\infty}$-relations~\eqref{aux3} and~\eqref{defs7} for an weighted algebra homomorphism and DA-bimodule respectively, it is clear that the operations~\eqref{aux6} make $\:^{\B} M_{\A}$ into a bona fide weighted DA-bimodule, which we denote by $\:^{\B} [\varphi]_{\A}$.

        \begin{remark}\label{aux7}
            \emph{It is clear that the identity DA bimodule is isomorphic to the DA-bimodule determined by the identity homomorphism, i.e.}
            \[
                \:^{\A}\id_{\A} \cong \:^{\A} [\id]_{\A}
            \]
            \emph{for any weighted $\A_{\infty}$-algebra $\A$.}
        \end{remark}

        There is a partial converse to the construction above. 
        
        \begin{proposition}\label{aux5} 
            [Analogous to Lemma 2.2.50 of~\cite{Bim}] Let $\A, \B$ be weighted $\A_{\infty}$ algebras over the same ground ring $R$ and weight space $\Lambda$. Assume that $R$ is a polynomial ring over the ring of idempotents with $N$ generators, as defined above, and that $\A$ and $\B$ satisfy the conditions of Definition~\ref{aux28}, and let $\:^{\B} M_{\A}$ be a weighted DA-bimodule with idempotents in the sense of Definition~\ref{aux32}.
            
            Suppose that the underlying $R$-bimodule of $M$ has one generator in each idempotent, and that $\delta^1_1= 0$. 
            Then there exists an $\A_{\infty}$-homomorphism $\varphi: \A \to \B$ with $M = \:^{\B}[\varphi]_{\A}$. 
        \end{proposition}

        \begin{proof}
            Let $\m_1, \ldots, \m_N$ denote the generators of $M$ with $\m_j$ in idempotent $j$ for each $1 \leq j \leq N$. Note that because $\B$ satisfies the compatibility conditions of Definition~\ref{aux28} and $M$ satisfies the compatibility conditions of Definition~\ref{aux32}, any non-zero element of $\B \otimes M$ can be written as a sum of elements of the form $b \otimes \m_i$, where the $b \in \B$ with final idempotent $i$. This means that there is a canonical isomorphism $f: \B \otimes M \to \B$, given by $f(b \otimes \m_i) = b$ for each $b \in \B$ with final idempotent $i$, extended linearly to all $\B \times M$. Now, define $\varphi$ as follows. For each $n \in \Z_{\geq 0}, \w \in \Lambda$, and $a_1, \ldots, a_n \in \A$ such that $a_1$ has initial idempotent $i$, write 
            \begin{equation}\label{aux8}
                \varphi_n^{\w} (a_1, \ldots, a_n) = (f\circ \delta_{n + 1}^{1, \w})(\x_i, a_1, \ldots, a_n)
            \end{equation}
            Since $\delta_1^1 \equiv 0$, $\varphi_0 = 0$, as required. That $\varphi$ satisfies the $\A_{\infty}$-relations~\eqref{aux3} follows immediately from~\eqref{aux8} and the $\A_{\infty}$-relations~\eqref{defs7} for the DA-bimodule $M$. That $\varphi$ satisfies the compatibility conditions of Lemma~\ref{aux29} follows from the fact that $\B$ satisfies the compatibility conditions of Definition~\ref{aux28} and $M$ satisfies the compatibility conditions of Definition~\ref{aux32}.
        \end{proof}

        \subsection{Box tensor products}\label{boxy}

        The aim of this section is to define the box tensor product of a DD- and a weighted AA-bimodule, both over weighted $\A_{\infty}$-algebras.
        \begin{definition}\label{box1}
            Box tensor products of DD- and AA-bimodules
        \end{definition}    
        \begin{itemize}
            \item Weighted $\A_{\infty}$-algebras $\A$ and $\B$, ground rings $R_1$ and $R_2$ respectively, and weight-spaces $\Lambda_1$ and $\Lambda_2$, respectively;

            \item We require that the weight spaces $R_1, R_2$ and ground rings $\Lambda_1, \Lambda_2$ are compatible in the sense of Definition~\ref{defs2};

            \item A weighted AA-bimodule $\:_{\B} Y_{\A}$;

            \item A DD-bimodule $\:^{\A} X^{\B}$; 

            \item A weighted algebra diagonal $\Gamma^{*,*}$

            \item A weighted bimodule diagonal primitive $p^{*,*,*}$ compatible with $\Gamma^{*,*}$;
        \end{itemize}

        The box tensor product $\:^{\A} X^{\B} \:_{\B}\boxtimes^{\B} \:_{\B} Y_{\A}$ is a DA-bimodule with underlying $R_1$-bimodule $X \otimes_{R_2} Y$, and basic operations given by 
        \begin{equation}\label{box2}
            \delta_{j + 1}^{1, \w} = 
            \sum_{\tiny\begin{matrix}
                n \geq 0 \\
                (n, j, \w) \neq (0,0,0) \\
                (S,T) \in \p^{n,j, \w}
            \end{matrix}} (\mu(S) \otimes id_X \otimes m(T) )\circ \delta^{n, X},
        \end{equation}
        so that {\small
        \begin{equation}\label{box3}
            \delta^{j,\w} = \sum_{\Cc} (\mu(S_k) \otimes \id_X \otimes m(T_k))\circ (\mu(S_{k - 1}) \otimes \delta^{n_k, X} \otimes m(T_{k - 1})) \circ \cdots \circ (\mu(S_1)  \otimes \delta^{n_2, X} \otimes m(T_1))\circ(\delta^{n_1, X} \otimes \id_Y \otimes \Delta)
        \end{equation}}
        where $\Cc$ denotes the set of conditions:
        \[
        \begin{matrix}
                k \geq 0 \\
                n_i, j_i \geq 0 \\
                \w_i \in \Lambda \\
                \sum \w_i = \w \\
                \sum_{i = 1}^k j_i = j \\
                (n_i, j_i, \w_i) \neq (0,0,0) \text{ for each }i \\
                (S_i, T_i) \in \p^{n_i, j_i, \w_i} \text{ for each }i
            \end{matrix}
        \]
        In terms of trees, this looks like
        \begin{center}
            $\delta^{k, \w} = \sum\limits_{\Cc} $ \hspace{.5cm}\parbox{6cm}{\includegraphics[width = 6cm]{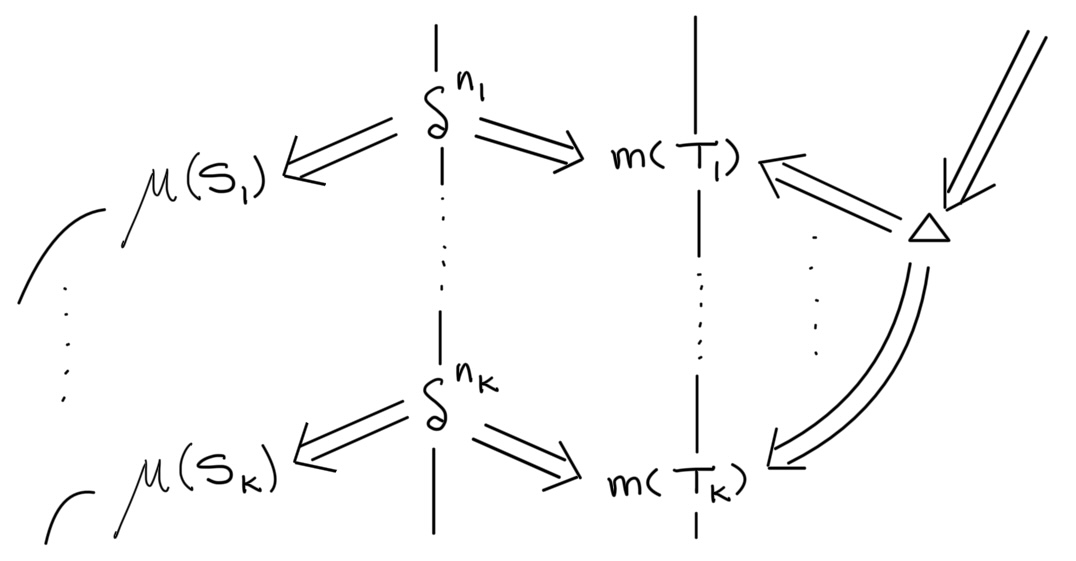}}
        \end{center}

        We next need to verify:

        \begin{lemma}\label{box4}
            With operations as in~\eqref{box2} and~\eqref{box3}, $X \boxtimes Y$ is a weighted DA-bimodule over $\A$ (on both left and right) in the sense of Definition~\ref{defs6}.
        \end{lemma}

        \begin{proof}
            We are going to use the expression for the operations given in~\eqref{box3}, and verify~\eqref{defs7}. We will do this calculation in terms of trees, since the arithmetic is easier to see in this form. 

            Notice first that for each $(n,j,\w) \neq (0,0,0)$, 
            \begin{equation}\label{box5}
                \del (\mu \otimes m )(\p^{n,j, \w}) = (\mu \otimes m ) (\del \p^{n,j,\w}) = 0
            \end{equation}
            because of the $\A_{\infty}$-relations for an $\A_{\infty}$-algebra and AA-bimodule. Note also that for each $n, j \geq 0, \w \in \Lambda$ with $(n,j, \w) \neq (0,0,0)$ 
            \begin{equation}\label{box6}
                \sum_{\tiny \begin{matrix}
                    n_1 + n_2 = n + 1 \\
                    \w_1 + \w_2 = \w
                \end{matrix}} (\mu \otimes m)(\p^{n_2, j, \w_2} \circ_L \gamma^{n_1, \w_1}) \circ (\delta^{n, X} \otimes \id_{Y\otimes \TT^* \A}) = 0
            \end{equation}
            by the $\A_{\infty}$-relations for the DD-bimodule $X$, which can be written as
            \[
                \sum_{\tiny \begin{matrix}
                    n \geq 0 \\
                    n_1 \leq n \\
                    \w_1 \in \Lambda
                \end{matrix}} (\mu(\gamma^{n_1, \w_1}) \otimes \id_X) \circ \delta^{n, X} = 0.
            \]
            This means that by the compatibility relation~\ref{defs10}, it follows that for each $n, j \geq 0, \w \in \Lambda$ with $(n,j,\w) \neq (0,0,0)$, we have
            \begin{align}\label{box7}
                \sum_{\tiny \begin{matrix}
                    \sum n_i = n \\
                    \sum j_i = j \\
                   \vv + \sum \w_i = \w
                \end{matrix}} (\mu \otimes m ) &(\CR^{\vv}(\p^{n_1,j_1,\w_1}, \ldots, \p^{n_k,j_k, \w_k})) \circ (\delta^{n,X} \otimes \id_{Y} \otimes \Delta) \notag \\
                & + \sum_{\tiny \begin{matrix} j_1 + j_2 = j + 1 \\ \w_1 + \w_2 = \w \end{matrix}}(\mu \otimes m ) (\p^{n, j_2, \w_2} \circ_R \Psi_{j_1}^{\w_1}) \circ (\delta^{n, X} \otimes \id_{Y} \otimes \Delta) = 0.
            \end{align}
            
            The first term of~\eqref{defs7} is
            \begin{center}
                \textbf{(*)}{$\sum\limits_{\w_1 + \w_2 = \w} \delta^{X \boxtimes Y, \w_2} \circ (\id_{X \boxtimes Y} \otimes \overline{D}^{\A, \w_1}) = \sum\limits_{\Cc'} \sum\limits_{\tiny \begin{matrix}1 \leq i \leq k \\ j_i' + j_i'' = j_i \end{matrix}}$} \parbox{8cm}{\includegraphics[width = 8cm]{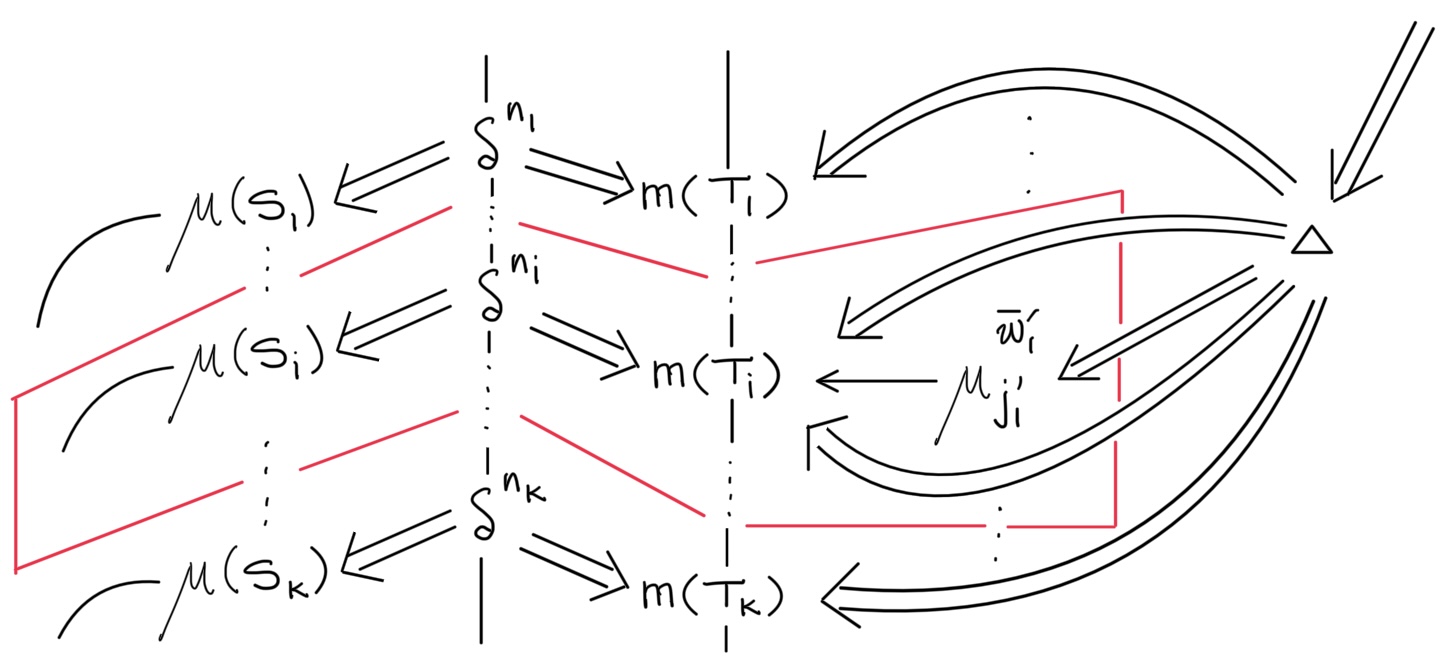}}
            \end{center}
            where $\Cc'$ denotes the set of conditions
            \[
                \begin{matrix}
                    k \geq 0 \\
                    n_{\ell}, j_{\ell} \geq 0 \text{ for each } \ell \\
                    \w_{\ell}'' \in \Lambda \\
                    \sum \w_{\ell}'' = \w_2 \\
                    (n_{\ell}, j_{\ell}, \w_{\ell}') \neq (0,0,0) \text{ for each } \ell \\
                    (S_{\ell}, T_{\ell}) \in \p^{n_{\ell}, j_{\ell}, \w_{\ell}''} \text{ for each } \ell
                \end{matrix}
            \]
            Notice that for each $i$, the portion of this term circled in red is of the form 
            \begin{align*}
                \sum_{\tiny \begin{matrix}  j_1 + j_2 = j + 1 \\ \w_1 + \w_2 = \w_0 \\(S, T) \in \p^{n, j_2, \w_2}\end{matrix}}(\mu(S) \otimes \id_{X} \otimes m(T)) \circ & (\id_{\TT^*\A \otimes X \otimes \TT^*\B \otimes Y} \otimes \mu_{j_1}^{\w_1})\circ (\delta^{n, X} \otimes \id_Y \otimes \Delta) \notag \\
                &= \sum_{\tiny \begin{matrix} j_1 + j_2 = j + 1 \\ \w_1 + \w_2 = \w_0 \end{matrix}}(\mu \otimes m) (\p^{n, j_2, \w_2} \circ_R \Psi_{j_1}^{\w_1}) \circ (\delta^{n, X} \otimes \id_{Y} \otimes \Delta) \tag{S1}\label{box8}
            \end{align*}
            for some fixed $n, j \geq 0, \w_0 \in \Lambda$ with $(n,j,\w_0) \neq (0,0,0)$. The second term is of the form
            \begin{center}
                \textbf{(**)} $\sum\limits_{\w_1 + \w_2 = \w} 
                (\overline{D}^{\A, \w_2} \otimes \id_{X \boxtimes Y}) \circ \delta^{X \boxtimes Y, \w_1} = \sum\limits_{\Cc''} \sum\limits_{\tiny \begin{matrix}
                    i \leq j \\
                \end{matrix}}$ \parbox{8cm}{\includegraphics[width = 8cm]{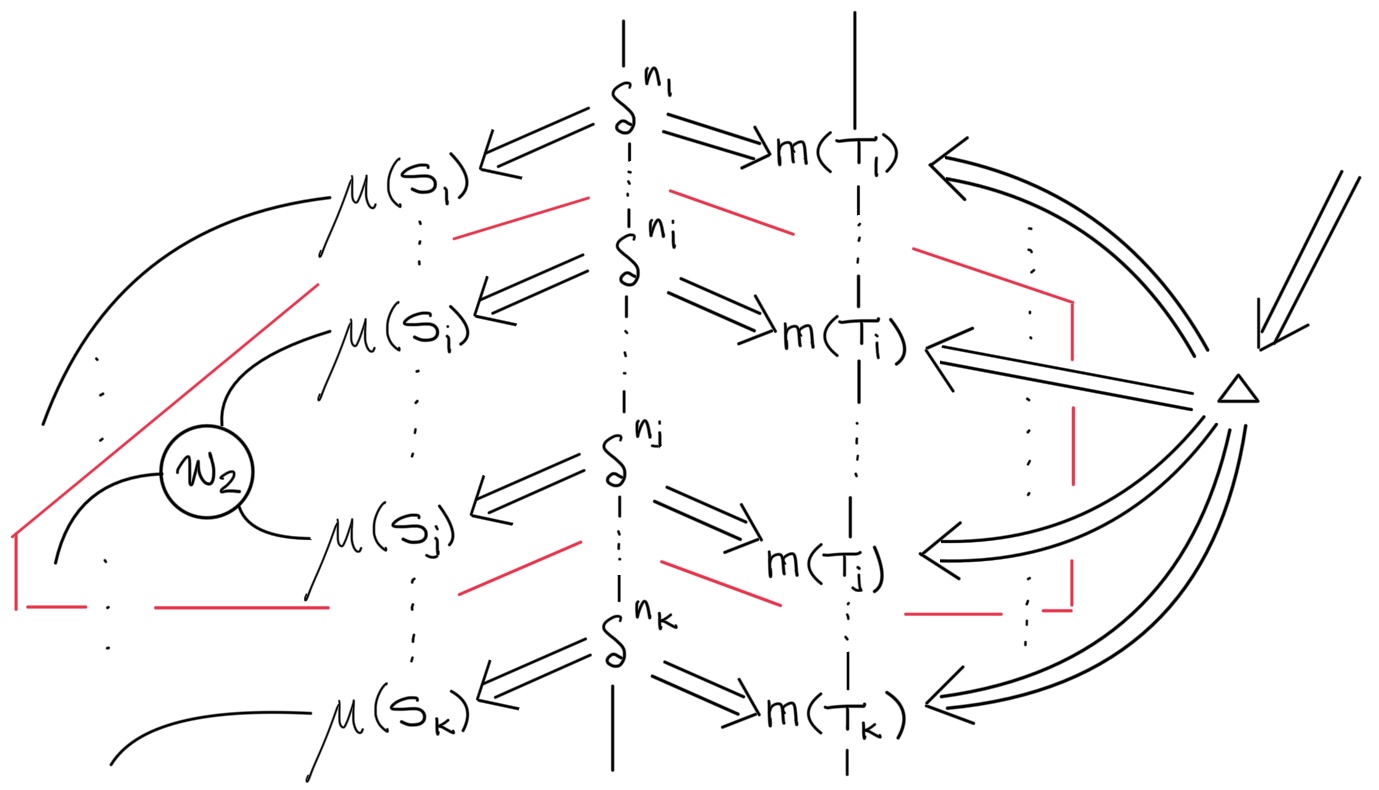}}
            \end{center}
            where $\Cc''$ denotes the set of conditions
            \[
                \begin{matrix}
                    k \geq 0 \\
                    n_{\ell}, j_{\ell} \geq 0 \text{ for each } \ell \\
                    \w_{\ell}'' \in \Lambda \\
                    \sum \w_{\ell}'' = \w_1 \\
                    (n_{\ell}, j_{\ell}, \w_{\ell}') \neq (0,0,0) \text{ for each } \ell \\
                    (S_{\ell}, T_{\ell}) \in \p^{n_{\ell}, j_{\ell}, \w_{\ell}''} \text{ for each } \ell
                \end{matrix}
            \]
            Notice that for each $i < j$, the portion of this term circle in red is of the form
            \begin{align*}
                \sum_{\tiny \begin{matrix}
                    \sum n_i = n \\
                    \sum j_i = j \\
                    \w_2 + \sum \w_i' = \w_0 \\
                    (S_i, T_i) \in \p^{n_i, j_i, \w_0} 
                    \end{matrix}} (\mu_n^{\w_2} \otimes \id_{X \otimes Y \otimes \TT^*\A}) & \circ (\mu(S_r) \otimes \id_X \otimes m(T_r)) \circ \cdots \circ (\mu(S_1) \otimes \id_X \otimes m(T_1)) \circ (\delta^{n, X} \otimes \id_Y \otimes \Delta) \\
                &=  \sum_{\tiny \begin{matrix}
                    \sum n_i = n \\
                    \sum j_i = j \\
                    \w_2 + \sum \w_i' = \w_0 \\
                \end{matrix}} (\mu \otimes m) ( \CR^{\w_2}(\p^{n_i, j_i, \w_i'}, \ldots, \p^{n_r, j_r, \w_r'})) \circ (\delta^{n, X} \otimes \id_Y \otimes \Delta) \tag{S2}\label{box9}
            \end{align*}
            for some fixed $(n, j, \w_0)\neq (0,0,0)$.

            Now, we can reindex the outer sums in (*) and (**), above, so that for each sequence of indices, the corresponding terms in (*) and (**) respectively are identical outside of the red circles. That is, we can rewrite~\eqref{box3} so that it is of the form
            \begin{center}
                $ \sum\limits_{\tilde{\Cc}} \sum\limits_{\begin{matrix}
                    n, j\geq 0 \\
                    \w_0 \in \Lambda \\
                    (n,j, \w_0) \neq (0,0,0)
                \end{matrix}}$ \parbox{6cm}{\includegraphics[width = 6cm]{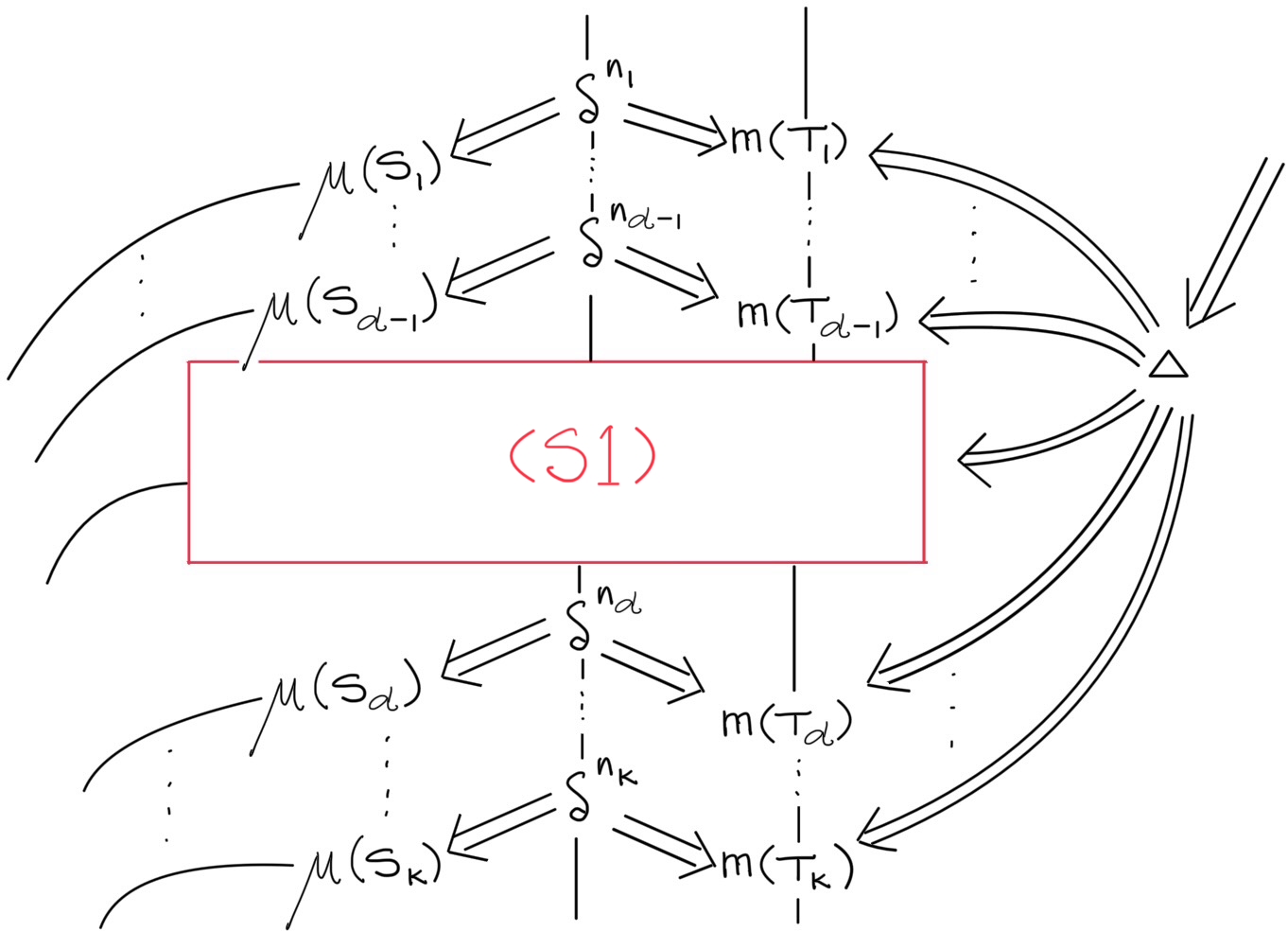}} $+$ \parbox{6cm}{\includegraphics[width = 6cm]{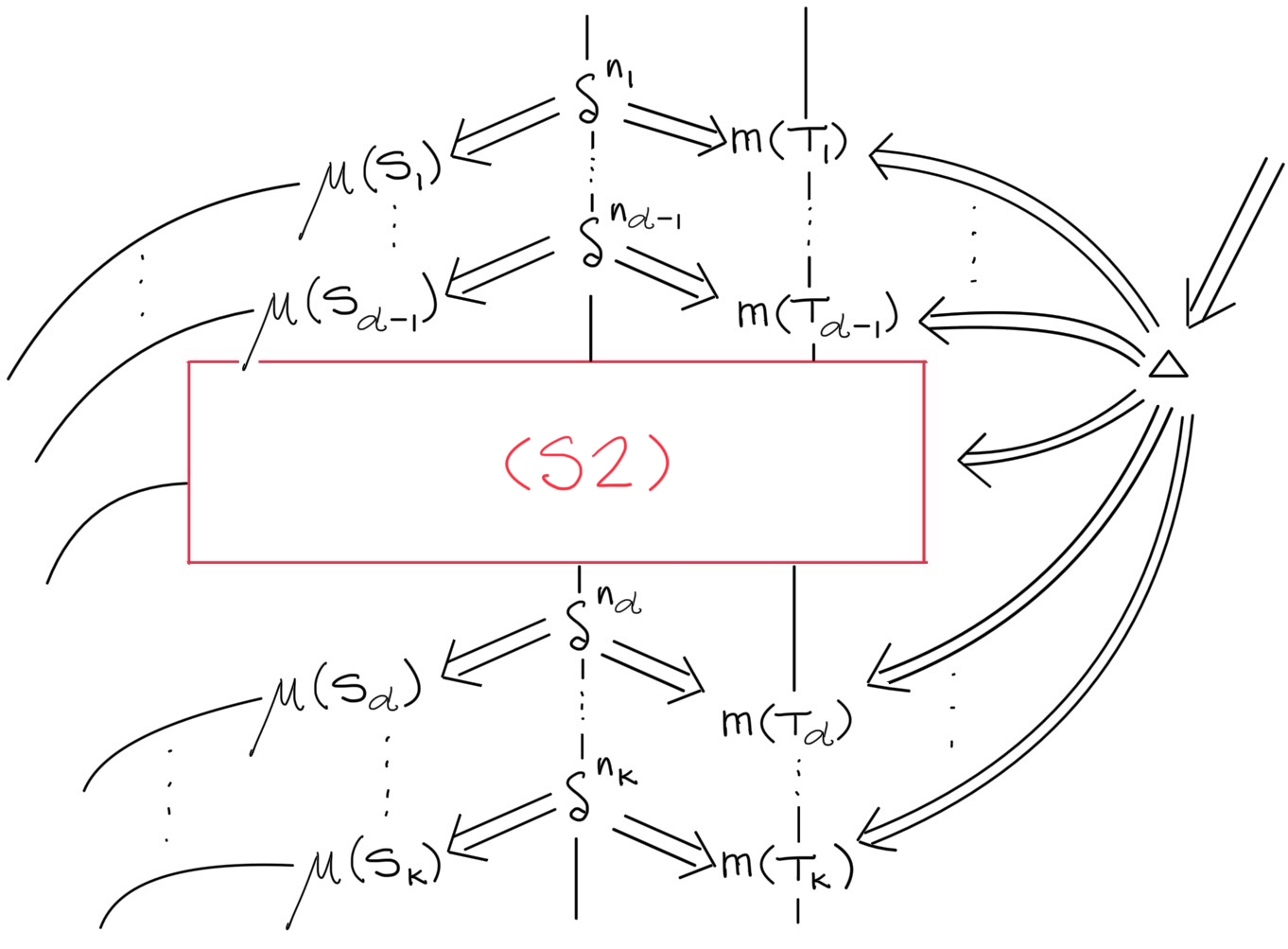}}
            \end{center}
            where (S1) and (S2) refer to the sums~\eqref{box8} and~\eqref{box9}, above, and $\tilde{\Cc}$ refers to the set of conditions
            \[
                \begin{matrix}
                    k \geq 0 \\
                    1 \leq \alpha \leq k \\
                    n_i, j_i \geq 0 \text{ for each } i \\
                    \w_i \in \Lambda \text{ for each } i \\
                    (S_i, T_i) \in \p^{n_i, j_i, \w_i} \text{ for each } i
                \end{matrix} 
            \]
            But by~\eqref{box7} we know that for each $n, j \geq 0, \w_0 \geq 0$ with $(n,j, \w_0) \neq (0,0,0)$, (S1) + (S2) $= 0$. This means that each term of the sum over $\tilde{\Cc}$ vanishes, and~\eqref{box3} is verified.
        \end{proof}

    \section{The $\A_{\infty}$-algebras $\A$ and $\B$}\label{algs}

        Fix $N \in \Z_{\geq 3}$, Let $R_0$ the ring of idempotents with $N$ generators, as defined in Definition~\ref{aux27}. Define 
         \begin{align*}
             R_1 &= R_0[V_0] \\
             R_2 &= R_0[V_1, \ldots, V_{N + 1}] \\
             \Lambda_1 &= \Z_{\geq 0} \langle\e_1, \ldots, e_{N + 1} \rangle \\
             \Lambda_2 &= \Z_{\geq 0}\langle \e_0 \rangle
         \end{align*}
         Then $(R_1, \Lambda_1), (R_2, \Lambda_2)$ form a compatible pair of ground rings and weight spaces in the sense of Definition~\ref{defs2}.

        \subsection{The construction of $\A$}\label{alg1}

         Now, let $\A$ be an $R_1$-bimodule generated by elements $\{U_i\}_{i = 1}^N, \{s_i\}_{i = 1}^N$, as well as an identity element $1$, with $R_1$-action defined as follows:
         \begin{align*}
             \I_j U_i = U_i \I_j &= \begin{cases}
                 U_i & i = j \\ 0 & \text{ otherwise}
             \end{cases} \\
             \I_j s_i = s_i \I_{j + 1} &= \begin{cases}
                 s_i & i = j \\ 0 &\text{ otherwise}
             \end{cases}
         \end{align*}
         We also stipulate that each of the $\I_i$ acts trivially on $1$, and that $V_0 a = a V_0$ for any $a \in \A$. We often call the $\{U_i\}_{i = 1}^N$ and $\{s_i\}_{i = 1}^N$ \emph{short chords}. 

         $\A$ is a bigraded space. We now define the two gradings and discuss how they interact with the operations. The first grading on $\A$ is a $\Z$-valued grading $m$ called the \emph{Maslov grading}. We define
         \[
            m(U_i) = m(s_i) = m(1) =  0 \text{ for each } 1 \leq i \leq N
        \]
        We also formally define
        \begin{align}\label{alg4}
            m(V_0) &= 2N + 2 \notag \\
            m(\I_j) &= 0 \text{ for each } 1 \leq i \leq N \notag \\
            m(\e_i) &= 2 \text{ for each } 1 \leq i \leq N + 1
         \end{align}
         and extend $m$ linearly to $R_1$ and $\Lambda_1$. Extend $m$ to all $\A$ according to the rule that 
        \begin{equation}\label{alg8}
            m(r \cdot a) = m(r) + m(a) \text{ for each } r \in R_1, \: a \in \A.
        \end{equation}
        The second grading on $\A$ is a vector-valued grading $A$ called the \emph{Alexander grading}, which maps generators of $\A$ into the lattice $\Z_{\geq 0}\langle \overline{1}, \ldots, \overline{2N} \rangle$ (where $\{\overline{i}\}$ denote the formal generators). We define $A$ on generators of $\A$ as
         \begin{align*}
            A(U_i) &= \overline{2i - 1} \text{ for each } 1 \leq i \leq N \\
            A(s_i) &= \overline{2i} \text{ for each } 1 \leq i \leq N \\
            A(1) &= \sum_{i = 1}^{2N} \overline{i}
         \end{align*}
         and formally set
         \begin{align}\label{alg5}
             A(V_0) &= \sum_{i = 1}^{2N} \overline{i} \notag \\
             A(\I_i) = A(\e_i) &= \overline{2i - 1} \text{ for each } 1 \leq i \leq N \notag \\
             A(\e_{N + 1}) &= \sum_{i = 1}^{N} \overline{2i}
         \end{align}
         extending $A$ linearly to all of $R_1$ and $\Lambda_1$. Extend $A$ to all $\A$ according to the rules that
         \begin{equation}\label{alg9}
            A(r \cdot a) = A(r) + A(a) \text{ for each } r \in R_1, \: a \in \A
         \end{equation}
         
         Operations $\mu_n^{\w}$ on $\A$ are defined in the following way. First, we define $\mu_2$, which we will write as multiplication:
         \begin{align*}
             U_i U_j &=  0 \text{ unless } i = j \\
            s_i U_j = U_i s_j &= \text{ for each } 1 \leq i, j \leq N \\
            s_i s_j &= 0 \text{ unless } i = (j - 1) \emm N
         \end{align*}
         Define
         \[
            s_{ij} = s_i s_{i + 1} \cdots s_{j-1}
         \]
         for each $i < j \emm N$, where all indices are written $\emm N$ -- so, for instance, $s_i$ can be written as $s_{i (i + 1)}$ , and $s_{ii} = s_i s_{i + 1} \cdots s_{i - 2} s_{i - 1}$ for each $i$. Define
         \[
            U_{N + 1} = \sum_{i = 1}^N s_{ii}
         \]
         Define
         \[
            \mu_0^{\e_i} = U_i \text{ for each } 1 \leq i \leq N + 1
         \]
         For each $1 \leq i \leq N$, write
         \[ 
            S_{2i-1} = (U_i, s_i, U_{i + 1}, s_{i + 1} \ldots, U_{i + N - 1}, s_{N + i + 1}),
        \]
        and
        \[
            S_{2i} = (s_i, U_{i + 1}, s_{i + 1}, \ldots, U_{i + N - 1}, s_{i + N - 1}, U_i),
        \]
        where all indices are $\emm N$. The non-zero unweighted higher operations are defined as follows:
         \[
            \mu_{j(2N - 2) +2}(S_i^j) = V_0^j I_i \cdot 1
         \]
         where $S_i^j$ denotes the concatenation of the string $S_i$ with itself in a certain sense. In the original construction of~\cite{KhKos}, these operations are determined by a family of $2N$-valent labeled graphs, of which two examples are given in Figure~\ref{alg10}. For a more complete statement of the necessary and sufficient conditions for a string $(a_1, \ldots, a_n)$ of elements of $\A$ to have $\mu_n^0(a_1, \ldots, a_n) \neq 0$, see Section 4.2 of~\cite{KhKos}, and more specifically, Theorem 4.1.

         \begin{figure}[H]
             \centering
             \includegraphics[width = 3cm]{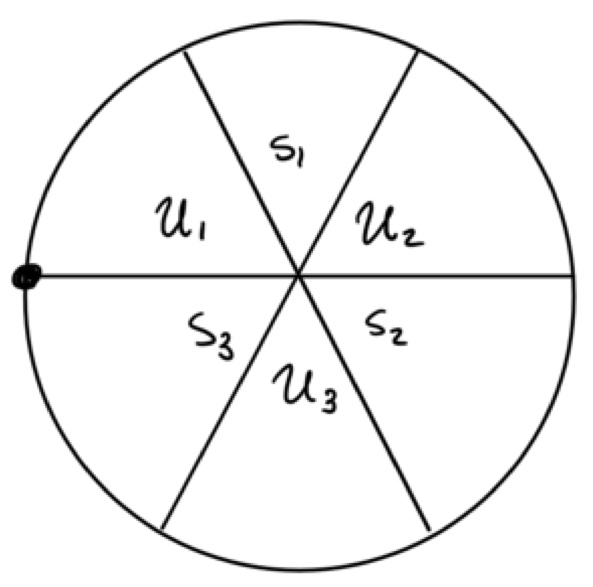} \hspace{.5cm}\includegraphics[width = 5cm]{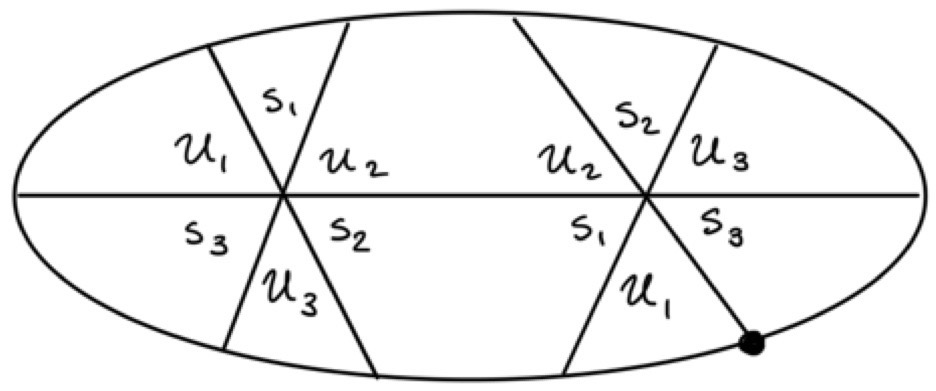}
             \caption{A $\mu_6$ and a $\mu_{10}$}
             \label{alg10}
         \end{figure}

         There are also infinitely many non-zero weighted operations $\mu_{j(2N - 2) - 2k + 2}^{\w}$, where $k = |\w|$ and $j \in \Z_{\geq 1}$. See Figure~\ref{alg11} for two such examples. For more complete conditions necessary and sufficient for a string $(a_1, \ldots, a_n)$ and $\w \in \Lambda_1 \smallsetminus \{0\}$ to have $\mu_n^{\w}(a_1, \ldots, a_n) \neq 0$, see Theorem 4.7 of~\cite{KhKos}.

         \begin{wrapfigure}{r}{3cm}
             \includegraphics[width = 3cm]{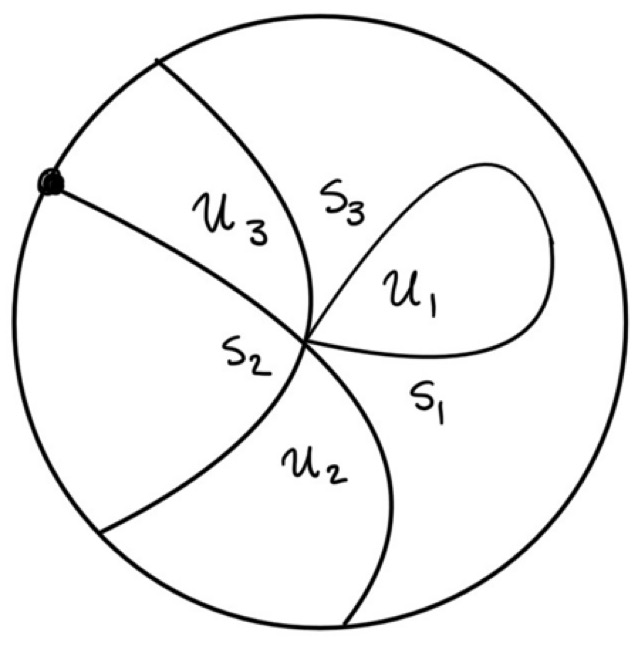}
             
             \includegraphics[width = 3cm]{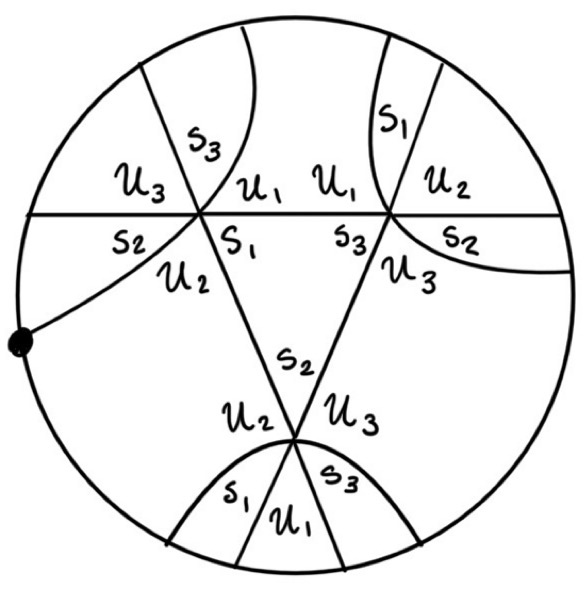}
             \caption{A $\mu_4^{\e_1}$ and a $\mu_{12}^{\e_4}$}\label{alg11}
         \end{wrapfigure}

          By construction, the operations on $\A$ satisfy
         \begin{equation}\label{alg3}
            m(\mu_n^{\w}(a_1, \ldots, a_n)) = \sum_{i = 1}^n m(a_i) + m(\w) + n - 2
         \end{equation}
          and
         \begin{equation}\label{aux34}
            A(\mu_n^{\w}(a_1, \ldots, a_n)) = \sum_{i = 1}^n A(a_i) + A(\w).
         \end{equation}

        \subsection{The construction of $\B$}\label{alg2}

        The weighted $\A_{\infty}$-algebra $\B$ is defined as an $R_2$-bimodule with generators $\{\rho_i\}_{i = 1}^N$ and $\{\sigma_i\}_{i = 1}^N$, and an identity element $1$, with $R_2$-action defined as follows:
        \begin{align*}
             \I_j \rho_i = \rho_i \I_j &= \begin{cases}
                 \rho_i & i = j \\ 0 & \text{ otherwise}
             \end{cases} \\
             \I_{j + 1} \sigma_i = \sigma_i \I_{j} &= \begin{cases}
                 \sigma_i & i = j \\ 0 &\text{ otherwise}
             \end{cases}
         \end{align*}
         We also stipulate that each $\I_i$ acts trivially on $1$, and that $V_i b = b V_i$ for each $1 \leq i \leq N + 1$ and $b \in \B$. We often call the $\rho_i$'s and $\sigma_i$'s \emph{short chords}.

         $\B$ is also equipped with a Maslov grading $m$ and an Alexander grading $A$. We define
         \begin{align*}
             m(\rho_i) = m(\sigma_i) &= -1 \text { for each } 1 \leq i \leq N \\
             m(1) &= 0 \\
         \end{align*}
         and formally define
         \begin{align}\label{alg6}
            m(\I_i) &= 0 \text{ for each } 1 \leq i \leq N \notag \\
            m(V_i) &= -2 \text{ for each } 1 \leq i \leq N + 1 \notag \\
            m(\e_0) &= - (2N - 2) \\
         \end{align}
         so we can extend $m$ linearly to $R_2$ and $\Lambda_2$. Extend $m$ to all $\B$ by setting
         \[ 
            m(b_1 \cdot b_2) = m(b_1) + m(b_2) \text{ for each } b_1, b_2 \in \B,
        \]
        and 
        \[
            m(r \cdot b) = m(r) + m(b) \text{ for each } r \in R_2, \: b \in \B.
        \]
        Define the Alexander grading $A$ as a map from the generators of $\B$ to $\Z_{\geq 0} \langle \overline{1}, \ldots, \overline{2N} \rangle$, as
         \begin{align*}
             A(\rho_i) &= \overline{2i - 1} \text{ for each } 1 \leq i \leq N \\
             A(\sigma_i) &= 
         \end{align*}
         and formally set
         \begin{align}\label{alg7}
             A(V_i) = A(\I_i) &= \overline{2i - 1} \text{ for each } 1 \leq i \leq N \notag \\
             A(V_{N + 1}) &= \sum_{i = 1}^N \overline{2i} \notag \\
             A(\e_0) &= \sum_{i = 1}^{2N}
         \end{align}
         Extend to all elements of $\B$ according to the rule that
         \[
            A(b_1 \cdot b_2) = A(b_1) + A(b_2) \text{ for each } b_1, b_2 \in \B
         \]
         and
         \[
            A(r \cdot b) = A(r) + A(b) \text{ for each } r \in R_2, \: b \in \B.
         \]

         \begin{remark}\label{algs1}
             \emph{For the Alexander grading on both $\A$ and $\B$, we make the following auxiliary definitions. First, we define}
             \[
                A(a_1, \ldots, a_n) = \sum_{i = 1}^n A(a_i)
             \]
             \emph{for $(a_1, \ldots, a_n)$ either a string of elements in $\A$ or a string of elements in $\B$. For $S$ denoting an element of $\A, \B, R_1, R_2, \Lambda_1, \Lambda_2$, or a string of elements in $\A$ or in $\B$, such that}
             \[
                A(S) = \sum_{i = 1}^{2N} k_i \cdot \overline{i}, 
             \]
             define
             \[
                |A(S)| = \sum_{i = 1}^{2N} |k_i|.
             \]
         \end{remark}
         
         We next discuss operations on $\B$. First, define $\mu_1(\rho_i) = V_i \cdot 1$ for each $i$. Define
         \begin{align*}
            \rho_i \sigma_j &= 0 \text{ unless } i = j + 1 \\
            \sigma_i \rho_i &= 0 \text{ unless } i = j \\ 
            \sigma_i \sigma_j &= 0 \text{ for each } i, j \\
            \rho_i \rho_j &= 0 \text{ for each } i, j
         \end{align*}
         We then write
         \[
            U_0 = \sum_{i = 1}^N \rho_{i + N - 1} \sigma_{i + N - 2} \cdots \rho_{i + 1} \sigma_{i} + \sum_{i = 1}^N \sigma_{i + N - 1} \rho_{i + N- 1} \cdots \sigma_i \rho_i
         \]
         Now define
         \[
            \mu_0^{\e_0} = U_0
         \]
         There are also various additional weighted and unweighted operations involving strings $(b_1, \ldots, b_n)$ with $|A(b_1, \ldots, b_n)| \geq N$. It is, however, difficult to give a simple and concise description for these operations, so we refer the interested reader to Section 5.2 of~\cite{KhKos}.

         We note only the following salient points. First, each element of $\B$ can be written uniquely as a sum of elements of the form $b = p(V_1, \ldots, V_{N + 1}) \cdot \tau$, where $p$ is a monomial in $V_1, \ldots, V_{N + 1}$, and $\tau$ is a product of $\rho_i$'s and $\sigma_i$'s. Second, the operations are constructed so that each non-zero operations satisfies
         \[
            m(\mu_n^{\w}(b_1, \ldots, b_n)) = \sum_{i = 1}^n m(b_i) + m(\w) + n - 2
         \]
         and
         \[
            A(\mu_n^{\w}(b_1, \ldots, b_n) = \sum_{i = 1}^n A(b_i) + A(\w).
         \]

    \section{The $\A_{\infty}$-bimodules}\label{bims}

        \subsection{The AA-bimodule $Y$}\label{aa}
        As defined in Section 6 of~\cite{KhKos}, the weighted AA-bimodule $\:^{\B} Y^{\A}$ is an $(R_2,R_1)$-bimodule with generators $\x_1, \ldots, \x_N$ with the following $R_1$ and $R_2$-actions: 
        \[
            \I_j \cdot \x_i = \x_i \cdot \I_j = \begin{cases}
                \x_i & \text{ if } i = j \\ 
                0 & \text{ otherwise} \\
            \end{cases}
        \]
        The bimodule $Y$ is also equipped with a $\Z$-valued Maslov grading $m$ and a $\Z_{\geq 0} \langle \overline{1}, \ldots, \overline{2N} \rangle$-valued Alexander grading $A$, which are defined as follows. Set $m(\x_i) = 0$ for each $i$, and $A(\x_i) = \overline{2i - 1}$ for each $i$. Extend $m$ to the ground rings $R_1$ and $R_2$ according to the rules of~\eqref{alg4} and~\eqref{alg6}, respectively, and extend $A$ to the ground rings $R_1$ and $R_2$ according to the rules of~\eqref{alg5} and~\eqref{alg7}, respectively. Then extend $m$ and $A$ to all $Y$ by rules analogous to~\eqref{alg8} and~\eqref{alg9}. 

        Non-zero operations are defined to be in one-to-one correspondence with rigid holomorphic disks in a star diagram such as those in Figure~\ref{intro1} with boundary on either the red $\alpha$-arcs, the blue $\beta$-arcs, or the boundary circles. For instance, in
        \begin{figure}
            \centering
            \includegraphics[width=8cm]{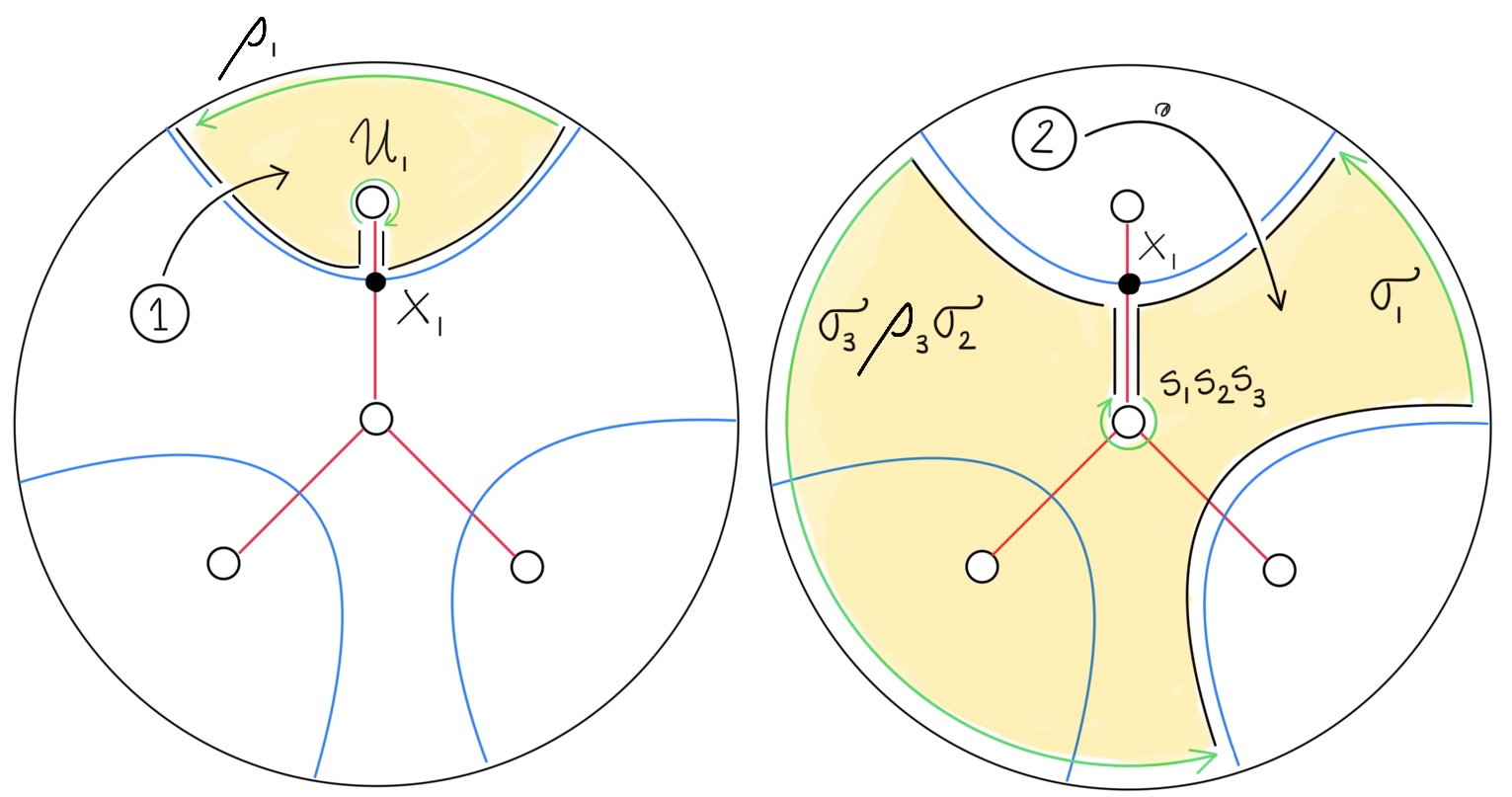}
            \caption{}
            \label{aa17}
        \end{figure}
        the disk labeled (1) corresponds to the operation 
        \begin{equation}\label{aa2}
            m_{1,1}^0(\rho_1, \x_1, U_1) = \x_1,
        \end{equation}
        and the disk labeled (2) corresponds to the operation
        \begin{equation}
            m_{2|1}^{\e_3}(\sigma_3\rho_3\sigma_2, \sigma_1, \x_1, s_{11}) = V_3 \cdot \x_1
        \end{equation}
        All operations satisfy the following grading rules:
        \begin{align*}
            m(m_{n|j}^{\w}(b_1, \ldots, b_n, \x, a_1, \ldots, a_j)) &= \sum_{i = 1}^n m(b_i) + m(\x) + \sum_{i = 1}^j m(a_i) + m(\w) + n + j - 1, \\
            A(m_{n|j}^{\w}(b_1, \ldots, b_n, \x, a_1, \ldots, a_j)) &= \sum_{i = 1}^n A(b_i) + A(\x) + \sum_{i = 1}^j A(a_i) + A(\w)
        \end{align*}
        The only other salient points about operations on $Y$ are that:
        \begin{itemize}
            \item $m_{n|0}^0 \equiv 0$ for each $n \geq 0$; This is clear for grading reasons, since the output of any non-zero operation on $Y$ is a generator of $Y$ (this is a fact of the construction in Section 6 of~\cite{KhKos}) and
            \[
                m( m_{n|0}^0(b_1, \ldots, b_n, \x)) = \sum m(b_i) + n - 1 < 0
            \]
            But all generators of $Y$ are in Maslov grading 0.

            \item If $a \in \A$ is a short chord, $m_{n|1}^0(b_1, \ldots, b_n, \x, a) \neq 0$ in precisely the following situations:
            \begin{itemize}
                \item $m_{1|1}^0(\rho_i, \x_i, U_i) = \x_i$ for each $1\leq i \leq N$;
                \item $m_{1|1}^0(\sigma_i, \x_i, s_i) = \x_i$ for each $1 \leq i \leq N$
            \end{itemize}
            This is less obvious, but follows from the construction in Section 6 of~\cite{KhKos}.
        \end{itemize}

        \subsection{The DD-bimodule $X$}\label{dd}

        As defined in Section 7 of~\cite{KhKos}, the DD-bimodule $\:^{\A} X\:^{\B}$ is defined as follows. As an $(R_1, R_2)$-bimodule, it has generators $\overline{\x}_1, \ldots, \overline{\x}_N$ with $R_1$ and $R_2$ action defined by
        \[
            \I_j \cdot \overline{\x}_i = \overline{\x}_i \cdot \I_j = \begin{cases}
                \overline{\x}_i & \text{ if } i = j \\ 
                0 & \text{ otherwise} \\
            \end{cases}
        \]
        The bimodule $X$ is equipped with Maslov and Alexander gradings as in all other cases, and has $m(\overline{\x}_i) = 0$, $A(\overline{x}_i) = \overline{2i - 1}$ for each $1 \leq i leq N$. We extend $m$ and $A$ to the ground rings, and then to all $X$ by means analogous to those used for $Y$. The differential is defined by
        \[
            \delta^1(\x_i) = U_i \otimes \x_i \otimes \rho_i + s_i \otimes \x_i \otimes \sigma_i.
        \]
        As in the case of $Y$, $\:^{\B} X^{\A}$ has the same generators, gradings, and $R_1$ and $R_2$-actions as $\:^{\A} X^{\B}$, except that the operations are reversed, i.e.
        \[
            \delta^{1, \:^{\B} X^{\A}}(\x_i) = \rho_i \otimes \x_i \otimes U_i + \sigma_i \otimes \x_i \otimes s_i.
        \]

        The last thing to show is that the sum in~\eqref{defs5} which makes up the $\A_{\infty}$-relations for $X$ is finite. For this, we use Lemma~\ref{aux35}. While $\A$ is not bonsai, we prove the following lemma in~\cite{KhKos}: 
        \begin{lemma}\label{aux36}
            [Lemma 7.1 of~\cite{KhKos}] Start with the generator $\x_i$ of $X$ which is in idempotent $\I_i$. Then the only non-vanishing $(\mu_n^{\w} \otimes \I) \circ \delta^n$ are those which give as output
		\begin{enumerate}[label = (\alph*)]
			\item $U_{i} \otimes V_i \otimes \overline{\x}$ for $1 \leq i \leq N$;
			
			\item $V_0 \otimes U_0\otimes \overline{\x}$;
		\end{enumerate}
		Moreover, each of these can be obtained in exactly two ways.
        \end{lemma}
        It follows from this lemma that $\delta^n \equiv 0$ for all $n$ sufficiently large, and hence, that the sum on the left hand side of~\eqref{defs5} is finite, and in fact, vanishes.

        \subsection{The box tensor product}

        The box tensor products $\:^{\A} X^{\B} \boxtimes \:_{\B} Y_{\A}$ and $\:^{\B} X^{\A} \boxtimes \:_{\A} Y_{\B}$ are defined according to the construction of Section~\ref{boxy}. The salient for the proof of Theorem~\ref{Koszul}, below, is that the each case the box product has a single generator. Indeed, consider the first case, i.e. $\:^{\A} X^{\B} \boxtimes \:_{\B} Y_{\A}$. Here, the underlying $(R_1,R_1)$-bimodule is a tensor product over $R_2$, so the generators are precisely products $\x \otimes \y$ where $\x$ is a generator of $X$ and $\y$ is a generator of $Y$. Hence, generators of the box product are among the $\{\overline{\x}_i \otimes \x_j\}_{1 \leq i, j \leq N}$. Now, for generator $\overline{\x}_i \otimes \x_j \in \:^{\A} X^{\B} \boxtimes \:_{\B} Y_{\A}$ and an idempotent $\I_k \in R_0$, we need
        \[
            (\overline{\x}_i \cdot \I_k) \otimes \x_j = \overline{\x}_i \otimes (\I_k \cdot \x_j)
        \]
        so we need $\x, \y$ to be in the same idempotent, i.e. $i = j$. Hence, the generators of $\:^{\A} X^{\B} \boxtimes \:_{\B} Y_{\A}$ are of the form $\{\overline{\x}_i \otimes \x_i\}_{i = 1}^N$, that is, there is one generator in each idempotent.  
        
    \section{Verifying Theorem~\ref{Koszul}}

    \begin{proof}[Proof of Theorem~\ref{Koszul}]
        There are two steps to the proof. First, we verify that
            \begin{equation}\label{dual1}
                \:^{\A} X^{\B} \boxtimes \:_{\B} Y_{\A} \cong \:^{\A}\id_{\A}.
            \end{equation}
            Then we show that
            \begin{equation}\label{dual2}
                \:^{\B} X^{\A} \boxtimes \:_{\A} Y_{\B} \cong \:^{\B}\id_{\B}
            \end{equation}

            To prove~\eqref{dual1}, we make the following observations. First, since the tensor product $X \otimes Y$ which determines the underlying $(R,R)$-bimodule for the box tensor product is taken over the idempotent ring, 
            $X \boxtimes Y$ has a single generator as an $(R,R)$-bimodule, namely $\x_0 = \sum_{j = 1}^N \I_j$. Also, $\delta^1_1 = 0$, since the terms from the right  hand side of~\eqref{box2}, for $j = 1, \: \w = 0$ would be of the form
            \begin{center}
                \begin{tikzcd}[column sep = tiny]
                    & \arrow[d, no head]&\arrow[dd, no head] \\ 
                    & \delta^{n, X} \arrow[dl, Rightarrow] \arrow[dd] \arrow[dr, Rightarrow] &  \\
                    \mu^{\A}(S) \arrow[d] & & m^{Y}(T) \arrow[d] \\
                    \: & \: & \:
                \end{tikzcd}
            \end{center}
            for $(S, T) \in \p^{n, 0, 0}$ for some $n$. But there is no non-zero, unweighted operation on $Y$ with no $\A$-inputs, so $\delta^1_1 = 0$. By Proposition~\ref{aux5}, there exists a weighted $\A_{\infty}$-algebra homomorphism $\varphi: \A \to \A$ such that $\:^{\A} [\varphi]_{\A} = \:^{\A} X^{\B} \boxtimes \:_{\B} Y_{\A} \cong \:^{\A}\id_{\A}$. 

            Next, we use the construction from the proof of Proposition~\ref{aux5} to compute $\varphi$ explicitly, and show that $\varphi = \id_{\A}$. The first step is to show that $\varphi_1: \A \to \A$ is the identity. For this, note that any term on the right hand side of~\eqref{box2} for $\delta_{2}^1$ is of the form
            \begin{center}
                \begin{tikzcd}[column sep = tiny]
                    & \arrow[d, no head]&\arrow[dd, no head]  & \arrow[ddl, bend left] \\
                    & \delta^{n, X} \arrow[dl, Rightarrow] \arrow[dd] \arrow[dr, Rightarrow] &  \\
                    \mu^{\A}(S) \arrow[d] & & m^{Y}(T) \arrow[d] & \\
                    \: & \: & \: &
                \end{tikzcd}
            \end{center}
            for $n \geq 0, (S,T) \in \p^{n, 1, 0}$. In particular, $m^Y(T)$ corresponds to a non-zero $m_{n|1}^0$ operation on $Y$. The only such operations with short-chord inputs on the $A$-side are 
            \begin{itemize}
                \item $m_{1|1}^0(\rho_i, \I_i, U_i) = \I_i$;

                \item $m_{1|1}^0(\sigma_i, I_i, s_i) = \I_i$;
            \end{itemize}
            Hence, for short-chord input, we need $n = 1$. Recalling that
            \begin{center}
                $\p^{1,1,0} = $ \parbox{2cm}{\includegraphics[width = 2cm]{prim6}}
            \end{center}
            and that
            \[
                \delta^{1, X}(\overline{\I}_i) = U_i \otimes \overline{\I}_i \otimes \rho_i + s_i \otimes \overline{\I}_i \otimes \sigma_i,
            \]
            it follows that 
            \[
                \delta_2^1(\x_0 \otimes U_i) = U_i \otimes \x_0,
            \]
            and 
            \[
                \delta_2^1(\x_0 \otimes s_i) = s_i \otimes \x_0,
            \]
            for each $i$. Hence, $\varphi_i = \id$ on the short chords in $\A$. Now, note that since the differential on $\A$ is zero, the $\A_{\infty}$-relation for $\varphi$ for two inputs is
            \begin{center}
               \parbox{2cm}{ \includegraphics[width = 2cm]{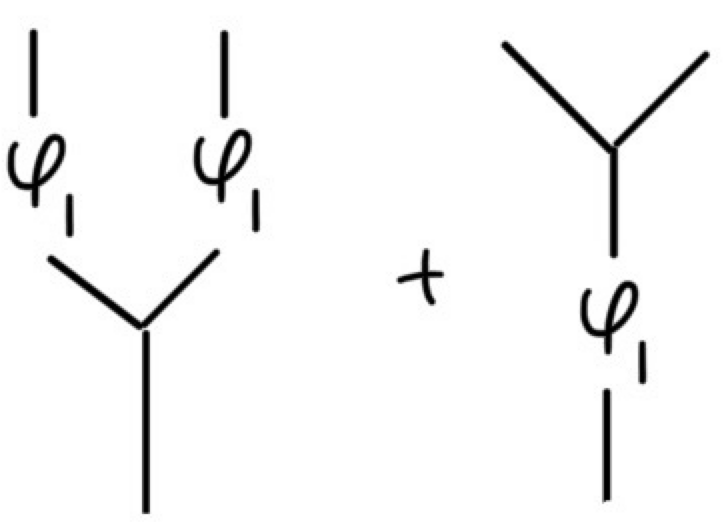}} \vspace{-0.5cm} $= 0$ \vspace{.5cm}
            \end{center}
            This means that 
            \[
                \varphi_1(U_i^2) = \varphi_1(U_i) \cdot \varphi_1(U_i) = U_i^2,
            \]
            so by induction, $\varphi_1(U_i^k) = U_i^k$ for each $1 \leq i \leq N$ and $k \in \Z_{\geq 0}$. Likewise,
            \[
                \varphi_1(s_i s_{i + 1}) = \varphi(s_i) \cdot \varphi(s_{i +1}) = s_{i} \cdot s_{i + 1},
            \]
            so again by induction, $\varphi_1(s_{ij}) = s_{ij}$ for each $i, j$. Hence, $\varphi_1$ is the identity on $\A$. 

            Next, we show that for grading reasons, $\varphi_i^{\w} = 0$ for each $(i, \w) \neq (1, 0)$. Recall that for $a_1, \ldots, a_n \in \A, \: \w \in \Lambda_1$, 
            \[
                m(\varphi_n^{\w}(a_1, \ldots, a_n) ) = \sum_{i = 1}^n m(a_i) + m(\w) + n - 1,
            \]
            so that since $m(a) = 0$ for each $a \in \A$, we have
            \[
                 m(\varphi_n^{\w}(a_1, \ldots, a_n) ) = m(\w) + n - 1.
            \]
            Since
            \[  
                m(\w) = 2 |\w|
            \]
            for each $\w \in \Lambda_1$, we would have 
            \[
                m(\varphi_n^{\w}(a_1, \ldots, a_n)) = 2|\w| + n - 1
            \]
           Since there are no elements with grading $> 0$, it follows that for $(n, \w)$ with $n \in \Z_{\geq 0}, \w \in \Lambda_1$ and $(n, \w) \neq (1, 0)$, $\varphi_n^{\w} \equiv 0$. Hence, $\varphi = \id_{\A}$, and it follows that
           \[
                \:^{\A}X^{\B} \boxtimes \:_{\B}Y_{\A} \cong \:^{\A}[\varphi]_{\A} = \:^{\A}[\id]_{\A} = \:^{\A} \id_{\A},
           \]
           as desired.

            We now turn our attention to~\eqref{dual2}. Let $\:^{\B} N_{\B}$ denote the bimodule on the left hand side of~\eqref{dual2}. By an argument identical to the one for~\eqref{dual1}, we can find a weighted $\A_{\infty}$-algebra homomorphism $\psi: \B \to \B$ such that 
            \[  
                \:^{\B} N_{\B} \cong \:^{\B}[\psi]_{\B}.
            \]
            By an argument identical to the one by which we showed $\varphi_1 = \id$ on the short chords, $\psi_1 = \id$ on the short chords in $\B$.
            
           We now show that $\psi_n^{\w} \equiv 0$ for all $(n, \w) \neq (1, 0)$ with $n \in \Z_{\geq 1}, \: \w \in \Lambda_2$. Recall that $\psi$ satisfies the following rules with respect to Maslov and Alexander gradings:
           \begin{equation}\label{dual3}
               m(\psi_n^{\w}(b_1, \ldots, b_n)) = \sum_{i = 1}^n m(b_i) + m(\w) + n - 1
           \end{equation}
            and
            \begin{equation}\label{dual4}
                A(\psi_n^{\w}(b_1, \ldots, b_n)) = \sum_{i = 1}^n A(b_i) + A(\w). 
            \end{equation}
            Recall that $V_1,\ldots V_{N + 1}$ are in the ground ring $R$ and $\psi$ is an $R$-module homomorphism, and that $\psi$ is unital (i.e. $\psi_n^{\w}(b_1 \ldots, \I_j, \ldots, b_n) = 0$ for any $\I_j$ and $(n, \w)$. Hence, we only need to compute values of $\psi_{n}^{\w}(b_1, \ldots b_n)$ where each $b_i$ is a product of $\rho_j$'s and $\sigma_j$'s. Note also that $\Lambda_2 = \Z_{\geq 0}\langle \e_{0} \rangle$, so for each $\w \in \Lambda_2,$ there exists $k \in \Z_{\geq 0}$ so that $\w = k \e_0$. Hence,~\eqref{dual3} becomes
            \begin{equation}\label{dual5}
                m(\psi_n^{k \e_0}(b_1, \ldots, b_n)) = \sum m(b_i) - k (2N - 2) + n - 1
            \end{equation}
            We are going to show that there is no element $b \in \B$ with Maslov index as in~\eqref{dual5}, and Alexander grading as in~\eqref{dual4}; in particular, we will show that any $b \in \B$ with Alexander grading given by the right hand side of~\eqref{dual4} will have Maslov index strictly less than~\eqref{dual5}.

            Any element $b \in \B$ is of the form $b = p(V_1, \ldots V_{N + 1}) \cdot \tau,$ where $p$ is a monomial in the $V_i$, and $\tau$ is a product of $\rho_j$'s and $\sigma_j$'s. Write the Alexander grading of $b$ as
            \[
                A(b) = \sum_{i = 1}^{2N} a_i \overline{i} = \sum_{i = 1}^N (\text{\# of $\rho_i$ appearing in $\tau$ + $\deg_{V_i} (p)$)} \cdot \overline{2i - 1} + \sum_{i = 1}^N (\text{\# of $\sigma_i$ appearing in $\tau$}) \cdot \overline{2i}. 
            \]
            and
            \begin{align}\label{dual7}
                m(b) &= - \sum_{i = 1}^N \text{\# of $\rho_i$ appearing in $\tau$} - 2 \sum_{i = 1}^N \deg_{V_i} (p) - \sum_{i = 1}^N \text{\# of $\sigma_i$ appearing in $\tau$} \notag \\
                &\leq -|A(b)|,
            \end{align}
            with equality if and only if $\deg(p) = 0$.
            
            Now look at a string $(b_1, \ldots, b_n)$, and a pair $(n, \w) \in \Z_{\geq 0} \times \Lambda_2$ where each $b_i$ is a product of $\rho_j$'s and $\sigma_j$'s. This means that
            \[  
                m(b_i) = - |A(b_i)| 
            \]
            for each $i$, so that
            \[
                \sum_{i = 1}^N m(b_i) = - |A(b_1, \ldots, b_n)|.
            \]
            Likewise,
            \[ 
                m(k \e_0) = - k (2N - 2) = |A(k \e_0)| + 2k
            \]
            This means that the right hand side of~\eqref{dual5} is 
            \begin{equation}\label{dual6}
                \text{RHS} = -|A(b_1, \ldots, b_n) + A(k \e_0)| + 2k + n - 1.
            \end{equation}
            
            If $b = \psi_n^{\w}(b_1, \ldots, b_n)$ is nonzero, then
            \[
                A(b) = A(b_1, \ldots, b_n) + A(k \e_0),
            \]
            so by~\eqref{dual7} 
            \[  
                m(b) \leq |A(b_1, \ldots, b_n) +  A(k \e_0)|
            \]
            but by~\eqref{dual5} and~\eqref{dual6}, this is a contradiction for any $(n, k \e_0) \neq (1, 0)$. This means that $\psi_n^{\w} \equiv 0$ for each $(n, \w) \neq (1,0)$. 

            It only remains to show that $\psi_1^0$ is the identity on long chords. But since we now know that $\psi_i^0 \equiv 0$ for each $i > 1$, we can now use the the $\A_{\infty}$-relation for two inputs to prove that $\varphi_1$ is the identity on all products of $\rho_j$'s and $\sigma_j$'s by an inductive argument identical to the one used for $\varphi,$ above.

            It now follows that $\psi = \id_{\B}$. Hence,
            \[
                \:^{\B} X^{\A} \boxtimes \:_{\A}Y_{\B} = \:^{\B}[\psi]_{\B} = \:^{\B}[\id]_{\B} = \:^{\B} \id_{\B},
            \]
            as desired.
    \end{proof}

    \bibliography{DiagRefs}
    \bibliographystyle{plain}

\end{document}